\documentclass{amsart}
\usepackage{amsthm}
\usepackage{amssymb}
\usepackage{amsmath}
\usepackage{amsmath,amscd}
\usepackage{xypic}
\usepackage{color}
\usepackage{hyperref}
\usepackage{tikz-cd}

\makeatletter
\@namedef{subjclassname@1991}{Subject}
\@namedef{subjclassname@2000}{Subject}
\@namedef{subjclassname@2010}{Subject}
\@namedef{subjclassname@2020}{Subject}
\makeatother

\newtheorem{thm}{Theorem}[section]
\newtheorem{cor}[thm]{Corollary}
\newtheorem{prop}[thm]{Proposition}
\newtheorem{lem}[thm]{Lemma}

\theoremstyle{definition}
\newtheorem{defn}[thm]{Definition}

\newtheorem{example}[thm]{Example}

\newtheorem{lemma}[thm]{Lemma}

\theoremstyle{remark}
\newtheorem{remark}[thm]{Remark}

\newcommand{\nc}{\newcommand}
\nc{\rnc}{\renewcommand}
\nc{\bb}[1]{{\mathbb #1}}
\nc{\bbA}{\bb{A}}\nc{\bbB}{\bb{B}}\nc{\bbC}{\bb{C}}\nc{\bbD}{\bb{D}}
\nc{\bbE}{\bb{E}}\nc{\bbF}{\bb{F}}\nc{\bbG}{\bb{G}}\nc{\bbH}{\bb{H}}
\nc{\bbI}{\bb{I}}\nc{\bbJ}{\bb{J}}\nc{\bbK}{\bb{K}}\nc{\bbL}{\bb{L}}
\nc{\bbM}{\bb{M}}\nc{\bbN}{\bb{N}}\nc{\bbO}{\bb{O}}\nc{\bbP}{\bb{P}}
\nc{\bbQ}{\bb{Q}}\nc{\bbR}{\bb{R}}\nc{\bbS}{\bb{S}}\nc{\bbT}{\bb{T}}
\nc{\bbU}{\bb{U}}\nc{\bbV}{\bb{V}}\nc{\bbW}{\bb{W}}\nc{\bbX}{\bb{X}}
\nc{\bbY}{\bb{Y}}\nc{\bbZ}{\bb{Z}}
\nc{\mbf}[1]{{\mathbf #1}}
\nc{\bfA}{\mbf{A}}\nc{\bfB}{\mbf{B}}\nc{\bfC}{\mbf{C}}\nc{\bfD}{\mbf{D}}
\nc{\bfE}{\mbf{E}}\nc{\bfF}{\mbf{F}}\nc{\bfG}{\mbf{G}}\nc{\bfH}{\mbf{H}}
\nc{\bfI}{\mbf{I}}\nc{\bfJ}{\mbf{J}}\nc{\bfK}{\mbf{K}}\nc{\bfL}{\mbf{L}}
\nc{\bfM}{\mbf{M}}\nc{\bfN}{\mbf{N}}\nc{\bfO}{\mbf{O}}\nc{\bfP}{\mbf{P}}
\nc{\bfQ}{\mbf{Q}}\nc{\bfR}{\mbf{R}}\nc{\bfS}{\mbf{S}}\nc{\bfT}{\mbf{T}}
\nc{\bfU}{\mbf{U}}\nc{\bfV}{\mbf{V}}\nc{\bfW}{\mbf{W}}\nc{\bfX}{\mbf{X}}
\nc{\bfY}{\mbf{Y}}\nc{\bfZ}{\mbf{Z}}
\nc{\bfa}{\mbf{a}}\nc{\bfb}{\mbf{b}}\nc{\bfc}{\mbf{c}}\nc{\bfd}{\mbf{d}}
\nc{\bfe}{\mbf{e}}\nc{\bff}{\mbf{f}}\nc{\bfg}{\mbf{g}}\nc{\bfh}{\mbf{h}}
\nc{\bfi}{\mbf{i}}\nc{\bfj}{\mbf{j}}\nc{\bfk}{\mbf{k}}\nc{\bfl}{\mbf{l}}
\nc{\bfm}{\mbf{m}}\nc{\bfn}{\mbf{n}}\nc{\bfo}{\mbf{o}}\nc{\bfp}{\mbf{p}}
\nc{\bfq}{\mbf{q}}\nc{\bfr}{\mbf{r}}\nc{\bfs}{\mbf{s}}\nc{\bft}{\mbf{t}}
\nc{\bfu}{\mbf{u}}\nc{\bfv}{\mbf{v}}\nc{\bfw}{\mbf{w}}\nc{\bfx}{\mbf{x}}
\nc{\bfy}{\mbf{y}}\nc{\bfz}{\mbf{z}}

\nc{\mcal}[1]{{\mathcal #1}}
\nc{\calA}{\mcal{A}}\nc{\calB}{\mcal{B}}\nc{\calC}{\mcal{C}}\nc{\calD}{\mcal{D}}
\nc{\calE}{\mcal{E}} \nc{\calF}{\mcal{F}}\nc{\calG}{\mcal{G}}\nc{\calH}{\mcal{H}}
\nc{\calI}{\mcal{I}}\nc{\calJ}{\mcal{J}}\nc{\calK}{\mcal{K}}\nc{\calL}{\mcal{L}}
\nc{\calM}{\mcal{M}}\nc{\calN}{\mcal{N}}\nc{\calO}{\mcal{O}}\nc{\calP}{\mcal{P}}
\nc{\calQ}{\mcal{Q}}\nc{\calR}{\mcal{R}}\nc{\calS}{\mcal{S}}\nc{\calT}{\mcal{T}}
\nc{\calU}{\mcal{U}}\nc{\calV}{\mcal{V}}\nc{\calW}{\mcal{W}}\nc{\calX}{\mcal{X}}
\nc{\calY}{\mcal{Y}}\nc{\calZ}{\mcal{Z}}
\nc{\fA}{\frak{A}}\nc{\fB}{\frak{B}}\nc{\fC}{\frak{C}} \nc{\fD}{\frak{D}}
\nc{\fE}{\frak{E}}\nc{\fF}{\frak{F}}\nc{\fG}{\frak{G}}\nc{\fH}{\frak{H}}
\nc{\fI}{\frak{I}}\nc{\fJ}{\frak{J}}\nc{\fK}{\frak{K}}\nc{\fL}{\frak{L}}
\nc{\fM}{\frak{M}}\nc{\fN}{\frak{N}}\nc{\fO}{\frak{O}}\nc{\fP}{\frak{P}}
\nc{\fQ}{\frak{Q}}\nc{\fR}{\frak{R}}\nc{\fS}{\frak{S}}\nc{\fT}{\frak{T}}
\nc{\fU}{\frak{U}}\nc{\fV}{\frak{V}}\nc{\fW}{\frak{W}}\nc{\fX}{\frak{X}}
\nc{\fY}{\frak{Y}}\nc{\fZ}{\frak{Z}}
\nc{\fa}{\frak{a}}\nc{\fb}{\frak{b}}\nc{\fc}{\frak{c}} \nc{\fd}{\frak{d}}
\nc{\fe}{\frak{e}}\nc{\fFf}{\frak{f}}\nc{\fg}{\frak{g}}\nc{\fh}{\frak{h}}
\nc{\fri}{\frak{i}}\nc{\fj}{\frak{j}}\nc{\fk}{\frak{k}}\nc{\fl}{\frak{l}}
\nc{\fm}{\frak{m}}\nc{\fn}{\frak{n}}\nc{\fo}{\frak{o}}\nc{\fp}{\frak{p}}
\nc{\fq}{\frak{q}}\nc{\fr}{\frak{r}}\nc{\fs}{\frak{s}}\nc{\ft}{\frak{t}}
\nc{\fu}{\frak{u}}\nc{\fv}{\frak{v}}\nc{\fw}{\frak{w}}\nc{\fx}{\frak{x}}
\nc{\fy}{\frak{y}}\nc{\fz}{\frak{z}}

\newcommand{\C}{{\mathbb C}}

\newcommand{\bP}{{\mathbb{P}}}

\newcommand{\Z}{{\mathbb Z}}
\newcommand{\Q}{{\mathbb Q}}
\newcommand{\caD}{{\mathcal D}}

\newcommand{\cO}{{\mathcal O}}

\newcommand{\Gr}{\mathrm{Gr}}

\newcommand{\csm}{c_{\mathrm{SM}}}
\newcommand{\ssm}{s_{\mathrm{M}}}

\newcommand{\one}{1\hskip-3.5pt1}

\DeclareMathOperator{\loc}{loc}

\DeclareMathOperator{\RHom}{RHom}
\DeclareMathOperator{\QH}{QH}
\DeclareMathOperator{\QK}{QK}
\DeclareMathOperator{\ev}{ev}
\DeclareMathOperator{\Mb}{\overline{\mathcal{M}}}
\DeclareMathOperator{\incl}{incl}

\begin{document}
\title[Left Demazure-Lusztig operators]{Left Demazure-Lusztig operators on equivariant (quantum) cohomology and K-theory}

\author{Leonardo C.~Mihalcea}
\address{
Department of Mathematics, 
Virginia Tech University, 
Blacksburg, VA 24061
USA
}
\email{lmihalce@vt.edu}

\author{Hiroshi Naruse}
\address{Graduate School of Education, University of Yamanashi, 
Kofu, 400-8510, Japan}
\email{hnaruse@yamanashi.ac.jp}

\author{Changjian Su}
\address{Department of Mathematics, University of Toronto,   
40 St. George St., Toronto, Ontario 
M5S 2E4
Canada}
\email{csu@math.toronto.edu}


\thanks{L.~C.~Mihalcea was supported in part by a Simons Collaboration Grant. H.~Naruse was supported in part by JSPS KAKENHI Grant Number 16H03921.}

\date{\today}
\begin{abstract} We study the Demazure-Lusztig operators induced by the left multiplication on partial flag manifolds $G/P$. We prove that they 
generate the Chern-Schwartz-MacPherson classes of Schubert cells (in equivariant cohomology), respectively their motivic Chern classes (in equivariant K-theory), in any partial flag manifold. Along the way we advertise many properties of the left and right divided difference operators in cohomology and K-theory, and their actions on Schubert classes. We apply this to construct left divided difference operators in equivariant quantum cohomology, and equivariant quantum K-theory, generating Schubert classes, and satisfying a Leibniz rule compatible with the quantum product. 
\end{abstract}

\maketitle


\section{Introduction} Let $G$ be a complex, semisimple, Lie group, and fix a Borel subgroup $B$, and a maximal torus $T:= B \cap B^-$, where $B^-$ is the opposite Borel subgroup. Fix also $P \supset B$ a standard parabolic subgroup. 

It has been known for some time that the Hecke algebra of the Weyl group $W:=N_G(T)/T$, and its degenerations (such as the nil-Hecke algebra), act on the cohomology and K-theory of flag manifolds $G/B$. What is perhaps less advertised is that each Hecke action is given by operators which come in a {\em left}, and a {\em right} flavor, with the two flavors playing a complementary role. For the equivariant cohomology ring $H^*_T(G/B)$, Knutson \cite{knutson:noncomplex} observed that the source of left/right flavors lies in two different automorphisms of $G/B$, given by the left and the right multiplication by elements of the Weyl group $W$.

The left $W$-action on a cohomology theory of $G/P$ is induced by the left multiplication by (representatives of) elements in $W$. This action is trivial in a {\em non-equivariant} cohomology theory.~The right $W$ action is only defined for cohomology theories of the full flag variety $G/B$, and it involves a non-algebraic map. First, one identifies $G/B = K/T_{\bb R}$, where $K \le G$ is a maximal compact subgroup with a (real) maximal torus $T_{\bb R}$. Then the right multiplication by elements of $W$ gives a well defined automorphism; see~\cite{knutson:noncomplex} or \S \ref{ss:leftrightW} below.

The left and right $W$ actions determine `left' and `right' divided difference operators, acting on (equivariant) cohomology/K-theory rings $H^*_T(G/B)$ and $K_T(G/B)$. The right operators are the classical BGG operators \cite{BGG} in cohomology, and Demazure 
operators \cite{demazure:desingularisations} in K-theory.~Algebraic versions of these operators were long used to define and study Schubert and Grothendieck polynomials, see e.g.~ \cite{lascoux.schutz:polynomes,lascoux:anneau}.~They were utilized in the model for the equivariant cohomology and equivariant K-theory of $G/B$ by Kostant and Kumar \cite{kostant.kumar:nil-Hecke,kostant.kumar:KT}; in their work, the left action is also present, via equivariant localization. The right operators may also be defined as push-pull operators $\pi_i^* (\pi_i)_*$, where $P_i$ is a minimal parabolic subgroup and $\pi_i: G/B \to G/P_i$ is the natural projection. Same definition extends equivariantly, and for more general flag manifolds \cite{kashiwara.shimozono:affine}. A different generalization, involving Demazure-Lusztig operators, is discussed below.

The left divided difference operators appeared in Brion's paper \cite{brion:eq-chow}, for the 
equivariant cohomology $H^*_T(G/B)$, and about the same time in Peterson's notes  \cite{peterson}, 
where they play a key role in the proof of the `quantum=affine' statement; see also \cite{lam.shimozono:quantum,kato:loop}. 
They were utilized to study the equivariant cohomology rings of flag manifolds, for instance, by providing recursions to calculate 
equivariant Schubert classes, their polynomial representatives, and Schubert structure constants; see e.g.~\cite{knutson:noncomplex,tymoczko:permutation,tymoczko:divided,IMN:double,ikeda.matsumura:pfaffian,tamvakis:double,lenart.zhong.etal:parabolic,goldin.knutson:schubert}. 

The left operators have two significant advantages: (1) they may be defined for an equivariant cohomology/K-theory ring
for {any `partial flag manifold'} $G/P$, and further, for any $G$-variety $X$; (2) they lift naturally to the quantum versions of these rings. The two 
properties may be traced to the fact that the left multiplication on $G/P$ by elements of $W$ 
induces a ring automorphism
of the appropriate equivariant quantum cohomology/K-theory of $G/P$. 

In this paper we study the left versions of the BGG, Demazure, and the Demazure-Lusztig (DL) operators
acting on $H^*_T(G/P)$ and $K_T(G/P)$ and their quantum versions.
They generate various degenerations of the Hecke algebra: the nil-Hecke algebra (for BGG operators), the $0$-Hecke algebra (for Demazure operators), and the degenerate, respectively 
the full Hecke algebra, for the DL operators. 

We show that via their geometric action, the left operators generate recursively the (appropriately deformed) Schubert classes for $G/P$, starting from the class of the point $1.P$. Equivalently, each of the equivariant rings $H^*_T(G/P)$,$K_T(G/P)$, $\QH^*(G/P)$, $\QK_T(G/P)$ is a cyclic module under the action of the appropriate Hecke algebra, compatible with a particular deformation of the Schubert basis. 

For (non-quantum) equivariant cohomology, and left BGG operators, this was proved 
in~\cite{brion:eq-chow,knutson:noncomplex}; see also \cite{peterson}.~This is also known 
more generally for an algebraic model of any (non-quantum) equivariant oriented 
cohomology theory of $G/P$, and for Schubert classes defined in that model, 
considered by Lenart, Zainoulline and Zhong in \cite{lenart.zhong.etal:parabolic}. 
In this note we utilize a different approach, geometric, and deformations of Schubert
classes such as the Chern-Schwartz-MacPherson classes, or motivic Chern classes 
of Schubert cells, which do not appear in {\em loc.~cit.} We expect that eventually 
there will be a `dictionary' from this note to \cite{lenart.zhong.etal:parabolic}, but 
this will be considered elsewhere.~Still, we added an Appendix where we showed, 
for an arbitrary $G$-variety $X$, that the left Demazure operators from this paper 
coincide with operators defined by Harada, Landweber and Sjamaar 
\cite{harada.landweber.sjamaar:divided}, and also equal to certain convolution 
operators. {The convolution approach is closely related to the setup 
from \cite{lenart.zhong.etal:parabolic}. 

The symmetry between 
the left and right actions becomes particularly vivid for $X=G/B$. First, the left/right 
Demazure operators correspond to left/right convolution operators. Further, in this case, 
the Atiyah-Hirzebruch map gives an isomorphism of $R(T)$-algebras 
$K_T(G/B) \simeq R(T) \otimes_{R(G)} R(T)$, where $R(-)$ is the representation ring. 
Under this isomorphism, the left/right actions of $W$ correspond to the actions of $W$ 
on the first/second factors, and the left/right Demazure operators are induced by 
actions of a single operator $\partial_i' \in End_{R(G)}(R(T))$ on the two factors. 
These facts seem well known to
experts, but we could not find a reference, and are also included in the Appendix.} 

Next we illustrate our results for the left DL operators. The cohomological left DL operators, acting on $H^*_T(G/P)$, are products of operators $\mathcal{T}_i^L$, indexed by simple reflections $s_i \in W$, and defined by:  
\[ \mathcal{T}_i^L:= -\delta_i+ s_i^L \/;\] see \S \ref{ss:cohDL}. This definition 
closely matches the one for the right operators, but utilizes the left Weyl group action $s_i^L$ by a simple reflection, and 
the left divided difference/BGG operator $\delta_i$. These operators satisfy the quadratic relations $(\mathcal{T}_i^L)^2 = id$, and the braid relations, thus giving a twisted representation of $W$ on $H^*_T(G/P)$ (sometimes referred as the action of the degenerate Hecke algebra). In particular, there are well defined operators $\mathcal{T}_w^L:=\mathcal{T}_{i_1}^L \cdot \ldots \cdot \mathcal{T}_{i_k}^L$ obtained from any reduced expression $w= s_{i_1} \cdot \ldots \cdot s_{i_k}$.

The right version of these operators, denoted by $\mathcal{T}_w^R$, was studied by Ginzburg \cite{ginzburg:methods} and algebraically, for $G=\mathrm{SL}_n$, by Lascoux, Leclerc and Thibon \cite{LLT:twisted}. More recently, it was shown in \cite{aluffi.mihalcea:eqcsm} that the right operators can be used to generate recursively the equivariant Chern-Schwartz-MacPherson (CSM) classes $\csm(X_w^\circ) \in H^*_T(G/B)$ (see \cite{macpherson:chern,ohmoto:eqcsm}) of Schubert cells $X_w^\circ \subset G/B$. For any simple reflection $s_i \in W$ and any $w \in W$, 
\begin{equation}\label{E:introrightDL} \mathcal{T}_i^R (\csm(X_w^\circ)) = \csm(X_{ws_i}^\circ) \quad \in H^*_T(G/B) \/. \end{equation} 
This formula holds for the complete flag variety $G/B$, since the right operators are only defined in this case. Our first main result is the analogue of \eqref{E:introrightDL} using the left operators $\mathcal{T}_i^L$, and which holds for any partial flag manifold $G/P$. Let $W_P$ is the subgroup of $W$ generated by reflections in $P$, and let $X_{wW_P}^{\circ}:=BwP/P\subset G/P$ be the Schubert cell in $G/P$. 
\begin{thm}[Theorem \ref{thm:csmleft}]\label{thm:main1} For any simple reflection $s_i \in W$ and any $w \in W$, the following holds:
\[ \mathcal{T}_i^L \csm(X_{wW_P}^{\circ}) = \csm(X_{s_iwW_P}^{\circ}) \in H_T^*(G/P).\]
\end{thm} 
There are similar identities for the opposite Schubert cells, and for the Poincar{\'e} duals of CSM classes of Schubert cells - the Segre-MacPherson classes;~cf.~Theorem \ref{thm:csmleft}.

For $G=\mathrm{SL}_n$, the left operators $\mathcal{T}_i^L$ are implicitly used by Rim{\'a}nyi, Tarasov and Varchenko 
\cite{RTV:partial} to give recursions for the Maulik-Okounkov stable envelopes \cite{maulik2019quantum}. 
By results in \cite{rimanyi.varchenko:csm,AMSS:shadows} stable envelopes are equivalent to CSM classes. 
In fact, in our language, the `R-matrix recursion' for stable envelopes from \cite{RTV:partial}, 
is simply the left $W$-action on CSM classes; see \S \ref{sec:Rmat}. Outside type A, the definition
of left operators seems to be new.

All this may be upgraded from the degenerate to the full Hecke algebra, and from cohomology to the equivariant K-theory.
The left K-theoretic Demazure-Lusztig operator for a simple reflection $s_i \in W$ is defined by:
\[ \mathcal{T}_i^L := \frac{1+ y e^{-\alpha_i}}{1-e^{-\alpha_i}} s_i^L - \frac{1+y}{1-e^{-\alpha_i} } \/. \] These 
act on $K_T(G/P)[y]$, with $y$ a formal variable; see \S \ref{ss:leftDL}. Here $\alpha_i$ is a simple root, $e^{-\alpha_i}$ is 
a character in $K_T(pt)$, and $s_i^L$ is the (K-theoretic) left $W$ action induced by left multiplication by 
$s_i$. The left operators satisfy the quadratic relations $(\mathcal{T}_i^L + id)(\mathcal{T}_i^L + y) = 0$, and the 
braid relations for $W$, yielding well defined operators $\mathcal{T}_w^L$ for any $w \in W$. 
In type A, these operators are implicit in constructions from 
\cite{RTV:Kstable}, where the authors study the K-theoretic generalization
of stable envelopes. In general, this definition seems to be new, although it closely 
resembles the one of the right DL operators $\mathcal{T}_w^R$ defined in \cite{lusztig:eqK}, see also 
\cite{brubaker.bump.licata,lee.lenart.liu:whittaker}. 

For flag manifolds, the K-theoretic stable envelopes are equivalent to (equivariant) {\em motivic Chern classes}
of Schubert cells $MC_y(X_{w W_P}^\circ) \in K_T(G/P)[y]$ defined in \cite{brasselet.schurmann.yokura:hirzebruch}; see \cite{feher2021motivic,AMSS:motivic} and details below. It was proved in \cite{AMSS:motivic} that for $w \in W$ and $s_i$ a simple reflection, the following holds in $K_T(G/B)[y]$: \[ \mathcal{T}_i^R MC_y(X_w^\circ) = \begin{cases} MC_y(X_{ws_i}^\circ) & \textrm{if } ws_i > w \/; \\ -(1+y) MC_y(X_w^\circ) -yMC_y(X_{ws_i}^\circ) & \textrm{otherwise} \/. \end{cases} \]
{Again, this identity only holds for the complete flag variety $G/B$. Our second main result is parabolic analogue of this, using the K-theoretic left operators $\mathcal{T}_i^L$.
Let $W^P$ be the set of minimal length representatives for the cosets in $W/W_P$.
\begin{thm}[Theorem \ref{thm:LDLMCP}]
For any $w\in W^P$ and any simple reflection $s_i$, the following holds in $K_T(G/P)[y]$:
\[\mathcal{T}_i^L(MC_y(X_{wW_P}^\circ))=\begin{cases} (-y)^{\ell(s_iw)-\ell(s_iwW_P)}MC_y(X_{s_i wW_P}^\circ) &\textrm{if } s_iw>w;\\
-(1+y)MC_y(X_{wW_P}^\circ)-y MC_y(X_{s_iwW_P}^\circ) &\textit{otherwise}. \end{cases}\]
\end{thm}
}   
In addition, we prove similar results where motivic classes are replaced by their Poincar{\'e} duals (these are the Segre motivic classes), or where we replaced $\mathcal{T}_i^L$ by certain `dual' DL operators; {see Theorem \ref{thm:SMCR} and Theorem \ref{thm:SMCTL}}. 

Although we focus on the left operators, the right operators play a central role in our proofs. We illustrate this by giving the basic idea of proving Theorem \ref{thm:main1} above. Let $\pi:G/B \to G/P$ be the projection, and $\pi_*: H^*_T(G/B) \to H^*_T(G/P)$ the induced map. By functoriality of CSM classes, and because $\pi$ is $G$-equivariant (thus $\pi_*$ commutes to $\mathcal{T}_i^L$), \[ \mathcal{T}_i^L \csm(X_{wW_P}^\circ) = \mathcal{T}_i^L \pi_* \csm(X_w^\circ) = \pi_*( \mathcal{T}_i^L(\csm(X_w^\circ))) \/. \] It remains to prove the $G/B$ case. In that case, by identity \eqref{E:introrightDL} for the right operators, $
\csm(X_w^\circ)= \mathcal{T}_{w^{-1}}^R[1.B]$, therefore, \[ \mathcal{T}_i^L(\csm(X_w^\circ)) = \mathcal{T}_i^L \mathcal{T}_{w^{-1}}^R [1.B] = \mathcal{T}_{w^{-1}}^R \mathcal{T}_i^L [1.B] \/, \] utilizing that left and right operators commute. The quantity $ \mathcal{T}_i^L [1.B] $ may be easily calculated directly in terms of CSM classes, and then one applies again \eqref{E:introrightDL} to obtain the desired identity.

Throughout the paper we tried to be as self contained as possible, and recall the basic facts needed about both left and right divided difference, Demazure, and Demazure-Lusztig operators. The survey of these formulae and properties of both the left and right operators may be of independent interest. 

We took the opportunity to show that the `classical' constructions of left divided difference and 
Demazure operators extend to the equivariant quantum cohomology ring $\QH^*_T(G/P)$ and the equivariant quantum K ring 
$\QK_T(G/P)$; see Propositions \ref{prop:quantumdelta} and \ref{prop:QKDem} below. 
This is clear if one regards both these rings as modules over the appropriate rings of (equivariant, quantum) parameters. 
The essential feature is that the left Weyl group multiplication induces automorphisms of the quantum rings, 
implying that the quantum extension satisfies a Leibniz rule involving the quantum product. 
We expect that this will be useful in finding `quantum' Schur and Schubert polynomials, representing Schubert classes in each ring. 
See examples \ref{ex:qhgr24} and \ref{ex:QKGr} below. Similar constructions, but from an algebraic point of view, were given 
by Ciocan-Fontanine \cite[App. J]{fulton.pragacz}, Kirillov and Maeno \cite{kirillov.maeno:quantum-schubert} and Lenart and Maeno \cite{lenart.maeno:quantum}.

{We mention that although we focus on the partial flag varieties $G/P$ in this note, the left Demazure operators $\delta_i$ and the left Demazure-Lusztig operators $\calT_i^L$ are well defined for the $T$-equivariant cohomology/K-theory of any complex variety with a $G$-action; in cohomology, this already appeared in Peterson's notes \cite{peterson}, see also \cite{harada.landweber.sjamaar:divided}. More details are given in the Appendix \ref{appendix}.

{\em Acknowledgements.} LM would like to thank A. Knutson for a conversation which sparked the definition of the quantum operators, and to T. Ikeda for useful comments and related collaborations. LM and CS would like to thank P. Aluffi, and J. Sch{\"u}rmann for related collaborations, and C. Zhong for useful comments. We thank the anonymous referees for useful suggestions and especially for precise pointers on how to expand the exposition on the symmetry between left and right actions and operators.

\section{Preliminaries}

\subsection{Equivariant cohomology} We recall some basic facts about the equivariant cohomology ring; see \cite{brion:eqint} for more details. 

Fix a smooth projective complex algebraic variety $X$ with a left action of a torus $T \simeq (\C^*)^r$. 
(Later $X=G/P$.) Throughout the paper we will use the equivariant Chow (co)homology theory developed in 
\cite{edidin.graham:eqchow}. There is an equivariant cycle map from the Chow to the ordinary (co)homology \cite[\S 2.8]{edidin.graham:eqchow}, and in the cases we will be interested in, this map is an isomorphism. Therefore
we follow the topological notation, although all our constructions are algebraic. 

We briefly recall the relevant definitions. Let $ET =(\C^\infty \setminus 0)^r$ be the universal $T$-bundle; 
it has a right $T$-action given by multiplication. Then the product $ET \times X$ 
has a right $T$-action given by $(e,y).t:= (et, t^{-1}y)$. The action is free, and the orbit space 
 $X_T:= (ET \times X)/T$ is called the Borel mixing space of $X$. By definition, 
 the equivariant cohomology ring  $H^*_T(X):= H^*(X_T)$. If $X=pt$, then $H^*_T(pt) = H^*((\bP^\infty)^r) = \C[t_1, \ldots , t_r ]$, where $t_i = c_1 (\cO_{(\bP^\infty)_i}(-1))$, the Chern class of the tautological line bundle pulled back from the $i$th component of $(\bP^\infty)^r$. The morphism $X_T \to BT:=ET/T$ gives $H^*_T(X)$ a structure of an algebra over $H^*_T(pt)$.

 Any $T$-stable subvariety $\Omega \subset X$ of complex codimension $k$ determines a class $[\Omega]_T \in H^{2k}_T(X)$. 
 If $Z$ is another smooth, projective variety and $f: X \to Z$ is a morphism, the push forward in homology and Poincar{\'e} duality determines a Gysin push-forward $f_*:H^i_T(X) \to H^{i+2(\dim Z -\dim X)}_T(Z)$, which is a $H^*_T(pt)$-module homomorphism. If $Z=pt$, we denote $f_*(a)$ by $\int_X a$. This gives the non-degenerate (intersection) pairing on $H^*_T(X)$ defined by 
\[ \langle a, b \rangle = \int_X a \cdot b  \quad \in H^*_T(pt) \/, \] for $a,b \in H^*_T(X)$. Here $\cdot$ denotes the intersection product in the equivariant Chow ring; under the cycle map, this corresponds to the cup product in the equivariant cohomology.

Let $\mathfrak{m}$ be the maximal ideal of $H^*_T(pt)$ generated by the equivariant parameters $t_i$. 
Since $H^*_T(pt)$ is a domain, $\mathfrak{m} \setminus 0$ is a multiplicative set, and we
define the {\em localized equivariant cohomology ring} $H^*_T(X)_{\loc}$ to be the localization 
$(H^*_T(X))_{\mathfrak{m} \setminus 0}$. We will identify $H^*_T(X)$ with a subring inside its localization.

Let $X^T= \{ x \in X: t.x = x ~\forall t \in T\}$ and assume that $X^T$ is finite.~By the localization theorem, see e.g.~\cite{brion:eqint}, 
$H^*_T(X)_{\loc}$ is a free $H^*_T(pt)_{\loc}$ module of rank equal to the number of fixed points in $X^T$. Further, a basis for
$H^*_T(X)_{\loc}$ is given by the fundamental classes $[x]_T$ of the fixed points $x \in X^T$. We call this the {\em fixed point basis}.
For each $x \in X^T$, let $i_x: \{ x \} \to X$ denote the inclusion. This is a $T$-equivariant proper morphism, and it 
induces the localization map $i_x^*: H^*_T(X) \to H^*_T(\{x\}) = H^*_T(pt)$. Then for any $\kappa \in H^*_T(X)$ we denote $\kappa|_x := i_x^*(\kappa)$. 
The restriction $i_x^*([x]_T)$ is the equivariant Euler class $e^T(T_x X)$, and it is an invertible element in $H^*_T(pt)_{\loc}$. In fact, if the $T$-module $T_x X$ is decomposed into weight spaces $T_x X = \oplus_i \C_{\lambda_i}$ then each weight $\lambda_i$ is a linear form in the equivariant parameters $t_i$ and $ e^T(T_x X) = \prod_i \lambda_i$. With this notation, any equivariant class $\kappa \in H^*_T(X)$ can be written uniquely as \begin{equation}\label{E:locexp} \kappa = \sum_{x \in X^T} \frac{\kappa|_x}{e^T(T_x X)} [x]_T  \quad \in H^*_T(X)_{\loc} \/. \end{equation}

\subsection{Preliminaries on flag manifolds}\label{sec:flags} Let $G$ be a complex, semisimple, Lie group, and fix 
$B,B^-$ a pair of opposite Borel subgroups. A maximal torus in $G$ is $T:= B \cap B^-$. Denote by 
$W:=N_G(T)/T$ the Weyl group, and by $\ell:W \to \mathbb{N}$ the associated length function. 
Denote also by $w_0$ the longest element in $W$; then $B^- = w_0 B w_0$. 

The datum $(G,B,T)$ determines a root system $R$ with positive roots $R^+$ and simple roots 
$\Delta := \{ \alpha_1, \ldots , \alpha_r \} \subset R^+$, such that the positive roots are in $B$. 
The simple reflection for the root $\alpha_i \in \Delta$ is denoted by $s_i$. Any subset $S$ of $\Delta$ determines a
parabolic subgroup $P:=P_S$ such that $B \subset P \subset G$ and $P/B$ corresponds to the simple roots in $S$. 
If $S = \{ \alpha_i \}$ then consists of a single element then $P_S$ is a minimal parabolic subgroup denoted by $P_i$.
If one fixes a parabolic group $P:=P_S$, then $W_P$ denotes the subgroup of $W$ generated by the simple reflections $s_i$ 
with $\alpha_i \in S$. Each element $w \in W$ has a unique decomposition $w = w_1 w_2$ where $w_2 \in W_P$ and 
$\ell(w) = \ell(w_1) + \ell(w_2)$. Then $w_1$ is the unique representative of minimal length for the coset $wW_P$ and 
one may define $\ell(wW_P) := \ell(w_1)$. We denote by $W^P$ the subset of $W$ containing all these minimal length 
representatives, and by $R_P^+$ the positive roots generated by the simple roots in $S$. 

Let $X:=G/B$ be the flag variety. This is a homogeneous space under the $G$ action. The $B$-orbits are the Schubert cells 
$X_w^\circ:= BwB/B \simeq \C^{\ell(w)}$, and the $B^-$-orbits are the opposite Schubert cells $X^{w,\circ}:=B^- w B/B$. 
The closures $X_w:= \overline{X_w^\circ}$ and $X^w:=\overline{X^{w,\circ}}$ are the Schubert varieties. With these definitions, 
$\dim_{\C} X_w = \mathrm{codim}_{\C} X^w = \ell(w)$. The Weyl group $W$ admits a partial ordering, called the Bruhat ordering, 
defined by $u \le v$ if and only if $X_u \subseteq X_v$. More generally, for any parabolic group $P$ one may consider the 
`partial flag manifold' $X^P:=G/P$. One defines in an analogous way the Schubert cells and the Schubert varieties: for $w \in W^P$, 
$X_{wW_P}^{\circ}:= BwP/P$ and $X_{wW_P} = \overline{BwP/P}$. Similarly $X^{wW_P,\circ} := B^- w P/P$, 
$X^{wW_P} := \overline{B^- w P/P}$ and then $\dim X_{wW_P} = \mathrm{codim}~X^{wW_P} = \ell(w)$. 
The Bruhat order from $W$ descends to one on its subset $W^P$. We often remove the $W_P$ from the notation 
if $w \in W^P$ and $P$ is understood from the context.

The $T$-fixed points of $X^P$ are in bijection with the elements of $W^P$; if $w \in W^P$ then we denote by $e_w$ the corresponding fixed point. 

\subsubsection{The Schubert basis and the fixed point basis} The equivariant cohomology ring $H^*_T(X^P)$ is a graded $H^*_T(pt)$-algebra with a $H^*_T(pt)$-basis given by the fundamental classes of the Schubert varieties:
\[ H^*_T(X^P) = \oplus_{w \in W^P} H^*_T(pt) [X_w] = \oplus_{w \in W^P} H^*_T(pt) [X^{w}] \/; \] 
(from here we remove $T$ from the notation $[\Omega]_T$, as all classes will be equivariant in this note). By the localization formula \eqref{E:locexp}, \[ [X^{w}] = \sum \frac{[X^{w}]|_{v}}{e(T_v(X^P))}[e_v] ~ \in H^*_T(X^P)_{\loc}\/, \] where the sum is over $v \in W^P$ such that $v \ge w$. As a $T$-module, the tangent space $T_v(X^P)$ has a weight decomposition, \[ T_v (X^P) = \bigoplus_{\alpha \in R^+ \setminus R_P^+} \C_{v(-\alpha)} \/, \] therefore the Euler class $e(T_v(X^P)) = \prod_{\alpha \in R^+ \setminus R_P^+} v(-\alpha)$.

Finally, let $\pi:G/B\rightarrow G/P$ be the natural projection. Then for any $w\in W$,
\begin{equation}\label{equ:pushXw}
\pi_*[X_{w}]=\begin{cases} [X_{wW_P}] & \textrm{ if } w\in W^P \/; \\ 0 & \textrm{ otherwise} \/. \end{cases}
\end{equation}

\section{Left and right Weyl group actions and operators}\label{ss:leftrightW} 
\subsection{Left and right Weyl group actions}\label{sec:twoWactions} In this section we follow mainly \cite{knutson:noncomplex}. The equivariant cohomology of the flag manifold $G/B$ admits a left and a right Weyl group action. The left action is given by the left multiplication on~$G/B$:~if $w \in W$, choose a representative $n_w \in N_G(T)$, and define the automorphism \[ \Phi_w: G/B \to G/B ; \quad xB \mapsto n_w x B \/. \] This is not an equivariant map, but it {\em is} equivariant with respect to the map $\varphi_w: T \to T$ sending $t \mapsto n_w t n_{w}^{-1}$. Therefore $\Phi_w$ induces a ring automorphism $w^L: H^*_T(G/B) \to H^*_T(G/B)$ defined by $w^L(a) = \Phi_w^*(a)$.\begin{footnote}{Note that non-equivariantly, this action is simply the identity.}\end{footnote} (It is easy to see that $w^L$ does not depend on the choice of the representative $n_w$.) If $\kappa \in H^*_T(pt)$ then $w^L(\kappa \cdot a) = w(\kappa) \cdot w^L(a)$, where $w(\kappa)$ is calculated from the natural action of $W$ on $H^*_T(pt) = \C[Lie(T)]$. Since the push-forward to a point is an equivariant morphism, it follows that for any $w \in W$ and $\eta_1,\eta_2\in H_T^*(X)$, we have
\[\langle w^L(\eta_1),w^L(\eta_2)\rangle=w.\langle\eta_1,\eta_2\rangle.\] Observe also that the left Weyl group action may be used to relate the $B$ and $B^-$ Schubert classes: if $w_0$ is the longest element in $W$, then from $\Phi_{w_0}^{-1} (X_w)= X^{w_0w}$ we deduce that \begin{equation}\label{E:w0B} w_0^L.[X_w] = [X^{w_0w}] \/. \end{equation}

The right $W$-action is induced by the right multiplication on $G/B \simeq K/T_{{\mathbb R}}$, where $K$ is the maximal compact subgroup of $G$, and $T_\bbR=K\cap B$. For any $w\in W$, let $w^R$ denote the automorphism of $H^*_T(G/B)$ induced by the right multiplication by (a representative of) $w^{-1}$ on $K/T_{{\mathbb R}}$. Since the right multiplication commutes with the (left) torus action, $w^R$ is a $H_T(pt)$-algebra automorphism. 

In terms of localizations, the automorphisms $w^R$ and $w^L$ are defined in the following way, cf.~ \cite{knutson:noncomplex}: for any $\eta\in H_T^*(G/B)$ and $u,w \in W$,
\[w^L(\eta)|_u=w(\eta|_{w^{-1}u})\quad \textit{ and } \quad w^R(\eta)|_u=\eta|_{uw}.\]
Notice that we have $w^Lu^L=(wu)^L$ and $w^Ru^R=(wu)^R$. 
We also record the following easy lemma.
\begin{lemma}\label{lemma:cohaction} (a) Let $w \in W$ and a simple reflection $s_i \in W$. Then \[ s_i^L [e_w] = [e_{s_i w}] \/; \quad s_i^R [e_w] = - [e_{ws_i}] \/, \] as classes in $H^*_T(X)$. 

(b) For any $u,v \in W$ we have $u^L v^R = v^R u^L$ as automorphisms of $H_T(X)$. 
\end{lemma}
\begin{proof} Let $X:=G/B$. We prove the claim by localization. For any $u\in W$,
\[s_i^L [e_w]|_u=s_i([e_w]|_{s_iu})=\delta_{w,s_iu}s_i(e(T_{w}X))=\delta_{s_iw,u}e(T_{s_iw}X)=[e_{s_i w}]|_u,\]
where $e(T_{w}X)=\prod_{\alpha>0}(-w\alpha)$ is the equvariant Euler class of the tangent fiber $T_wX$ at the fixed point $e_w$. This implies the first equality. Similarly, 
\[s_i^R [e_w]|_u=[e_w]|_{us_i}=\delta_{w,us_i}e(T_wX)=-\delta_{ws_i,u}e(T_{ws_i}X) = - [e_{ws_i}]|_u,\]
where we utilized that $e(T_wX)=-e(T_{ws_i}X)$. This finishes the proof of (a). Part (b) follows because the left and right multiplications on $G/B$ commute.\end{proof}

\subsection{Divided difference operators}\label{sec:cohdivdiff} One may utilize the left and right actions to define divided difference operators on $H^*_T(X)$. 

We start with the right divided difference operators, defined in \cite{BGG}, and utilized for a long time 
to study Schubert classes. 
Let $\alpha_i$ be any simple root, and let $\pi_i:G/B\rightarrow G/P_i$ be the projection. 
Then the {\em right} BGG operator 
is $\partial_i:=\pi_i^*\pi_{i*}$. In terms of localizations, it is given by \[(\partial_i\eta)|_v=\frac{\eta|_v-(s_i^R\eta)|_v}{-v\alpha_i},\]
where $v\in W$ and $\eta\in H_T^*(G/B)$. The operator $\partial_i$ is $H^*_T(pt)$-linear, and it satisfies $\partial_i^2 = 0$ and the usual 
braid relations. In particular, for any $w \in W$, there is a well defined operator $\partial_w$ acting on 
$H^*_T(G/B)$. It follows from the definition that the operator $\partial_i$ acts on Schubert classes by: 
\[ \partial_i [X_w] = \begin{cases} [X_{ws_i}] & \textrm{if } ws_i > w \/; \\ 0 & \textrm{otherwise} \/; \end{cases} \quad 
\partial_i [X^{w}] = \begin{cases} [X^{ws_i}] & \textrm{if } ws_i < w \/; \\ 0 & \textrm{otherwise} \/. \end{cases}\] 

The left divided difference operator, acting on $H^*_T(G/B)$, is defined by 
\begin{equation}\label{E:leftcohdiv} \delta_i = \frac{1}{\alpha_i}(id - s_i^L) \/, \end{equation} and it played a key role in the study of equivariant Schubert classes. These operators satisfy $\delta_i^2=0$ and the braid relations. In terms of Schubert classes
\begin{equation}\label{E:delschub} \delta_i [X^{w}] = \begin{cases} [X^{s_iw}] & \textrm{if } s_iw < w \/; \\ 0 & \textrm{otherwise} \/; \end{cases} \quad 
\delta_i [X_w] = \begin{cases} -[X_{s_iw}] & \textrm{if } s_i w > w \/; \\ 0 & \textrm{otherwise} \/. \end{cases}\end{equation}
The first equality is proved in \cite[\S 6.4]{brion:eq-chow} and in \cite[Prop.~2]{knutson:noncomplex}, and the second equality follows from the first by 
utilizing equation \eqref{E:w0B} and the fact that $w_0^L \delta_{\alpha_i} w_0^L = - \delta_{-w_0(\alpha_i)}$; observe that 
$w_0(\alpha_i)$ is a negative simple root. (To emphasize the dependence on simple root $\alpha_i$, we temporarily used 
the notation $\delta_{\alpha_i}$ instead of $\delta_i$.) With a different sign normalization, the formulae \eqref{E:delschub} 
also appeared in Peterson's notes \cite{peterson}. As we shall see below, one advantage of utilizing $\delta_i$ is that 
it descends to an operator on the equivariant cohomology of the partial flag manifold $G/P$; this is not true for $\partial_i$. 
 
We record some properties of the two operators; see \cite{knutson:noncomplex}.

\begin{lemma}\label{lemma:delcomm} (a) The left and right operators satisfy the following commutation relations, for any $i,j$:
\[ \delta_i s_j^R = s_j^R \delta_i \/; \quad \partial_i s_j^L = s_j^L \partial_i \/; \quad \delta_i \partial_j = \partial_j \delta_i \/. \]

(b) (Leibniz rule) For any $a,b \in H^*_T(G/B)$, \[ \delta_i(a \cdot b) = \delta_i(a) \cdot b + s_i^L(a) \delta_i(b) \/. \] In particular, for any $\lambda \in H^*_T(pt)$, $\delta_i(\lambda \cdot b) = \langle \lambda, \alpha_i^\vee \rangle b + s_i(\lambda) \delta_i(b)$. 

(c) For any $\kappa \in H^*_G(G/B)$, 
$\delta_i(\kappa \cdot a) = \kappa \delta_i(a)$.
\end{lemma}
\begin{proof} Part (a) follows from the fact that the left and right Weyl group actions commute. 
Parts (b) and (c) follow from immediate calculations.\end{proof}
For later use, we record the following corollary of the above Lemma.
\begin{cor}\label{cor:action}
For any simple root $\alpha_i$ and $w\in W$, the following holds in $H_T^*(G/B)$:
 \[ s^L_i ([X_w]) = \begin{cases} [X_{ w}]+\alpha_i [X_{s_iw}] & \textrm{if } s_iw > w \/; \\ [X_w] & \textrm{otherwise} \/. \end{cases} \]
\end{cor}
\begin{proof}
By Lemma \ref{lemma:delcomm}, we have
\[ s^L_i ([X_w])=s^L_i \partial_{w^{-1}}([X_{id}])  =\partial_{w^{-1}} s^L_i ([X_{id}]) =\partial_{w^{-1}} ([X_{id}]+\alpha_i [X_{s_i}]) \/.\] 
The last expression equals the right hand side of the claim.\end{proof}

\subsection{Cohomological Demazure-Lusztig operators}\label{ss:cohDL} In this section we recall the definition of the right cohomological Demazure-Lusztig (DL), 
and we define a left version of these operators. The right DL operators appeared in \cite{ginzburg:methods} in relation to degenerate Hecke algebras and in \cite{aluffi.mihalcea:eqcsm,AMSS:shadows} in the study of Chern-Schwartz-MacPherson classes. 

Fix $\alpha_i$ a simple root. The right DL operator, and the dual right DL operator are defined by 
\[ \mathcal{T}_i^R:= \partial_i - s_i^R \/; \quad \mathcal{T}_i^{R,\vee} = \partial_i + s_i^R \/. \] 
Both operators are $H^*_T(pt)$-linear, satisfy the braid relations and the quadratic relation 
$(\mathcal{T}_i^R)^2 = (\mathcal{T}_i^{R,\vee})^2 = id$. Thus they provide a twisted $W$-action on $H^*_T(G/B)$; see also \cite{LLT:twisted}.~The braid relations imply that there well defined operators 
$T_w^R, T_w^{R,\vee}$ defined as usual using reduced word decompositions. The duality refers to the fact that they are adjoint under the 
Poincar{\'e} pairing (see ~\cite{AMSS:shadows}): for any $a,b \in H^*_T(G/B)$, 
\begin{equation}\label{E:radj} \langle \mathcal{T}_i^R a, b \rangle = \langle a, \mathcal{T}_i^{R,\vee} b \rangle \/. \end{equation}

The {left} DL operators, and the dual left operators are defined by \[ \mathcal{T}_i^L:=-\delta_i + s_i^L \/; \quad \mathcal{T}_i^{L,\vee}:=\delta_i + s_i^L \/. \] It is easy to check that they satisfy the same braid and quadratic relations as the right DL operators. However, the left operators are not $H^*_T(pt)$-linear. We record next few properties of these operators, which follow easily from the properties of the left and right divided difference operators; we leave the details of the proofs to the reader.
\begin{lemma}\label{lemma:DLcomm} The following hold: 

(a) The left and right operators commute, i.e. \[ \mathcal{T}_i^L \mathcal{T}_j^R = \mathcal{T}_j^R \mathcal{T}_i^L; \quad \mathcal{T}_i^{L,\vee} \mathcal{T}_j^{R,\vee} = \mathcal{T}_j^{R,\vee} \mathcal{T}_i^{L,\vee} \/. \]

(b) For any $a,b\in H_T^*(G/B)$, $\langle \mathcal{T}_i^{L} (a),b \rangle=s_i .\langle a,\mathcal{T}_i^{L,\vee} (b)\rangle$.

(c) The left operator and its dual are related by $w_0^L\mathcal{T}_{\alpha_i}^{L}w_0^L=\mathcal{T}_{-w_0\alpha_i}^{L,\vee}$.
\end{lemma}

\subsection{Parabolic case}\label{sec:cohparabolic} Unlike the right versions of the BGG and DL operators, the left versions are well defined in the equivariant cohomology of any $G/P$, with the same definitions as for $G/B$. Consider the projection $\pi:G/B \to G/P$. The morphism $\pi$ is $G$-equivariant, and both the (Gysin) push-forward $\pi_*:H^*_T(G/B) \to H^*_T(G/P)$ and the pull back $\pi^*:H^*_T(G/P) \to H^*_T(G/B)$ are $H^*_T(pt)$-linear and commute with the left action $w^L$: for any $a \in H^*_T(G/B)$ and $b \in H^*_T(G/P)$,  \[ \pi_*(w^L.a) = w^L. \pi_*(a) \/; \quad \pi^*(w^L.b) = w^L. \pi^*(b) \/. \]
Further, $\pi^*:H^*_T(G/P) \to H^*_T(G/B)$ is an injective ring homomorphism, so the second equation may also be taken as the definition of $w^L$. All the left operators $\delta_i$, $\mathcal{T}_i^L$ and $\mathcal{T}_i^{L, \vee}$ commute with $\pi_*$ and $\pi^*$, and that the properties of the left BGG operators, and left DL operators and their duals, as stated in Lemmas \ref{lemma:cohaction},\ref{lemma:delcomm} and \ref{lemma:DLcomm}, hold in $H^*_T(G/P)$, with the same proofs. 

The action on Schubert classes is similar: if $w \in W^P$, and $s_i$ any simple reflection, then \[ \delta_i [X_{wW_P}] = \delta_i \pi_*[X_w] = \pi_* \delta_i [X_w] \/. \] By equations \eqref{E:delschub} and \eqref{equ:pushXw},
\begin{equation}\label{E:delPschub}\delta_i[X_{wW_P}]=\begin{cases} -[X_{s_iwW_P}] & \textrm{if } s_iw > w \textrm{ and } s_iw\in W^P \/; \\ 0 & \textrm{otherwise} \/. \end{cases}\end{equation} 
The same idea, but using $\pi^*$, may be utilized to calculate the action of $\delta_i$ on opposite Schubert classes $[X^{wW_P}]$:  
\begin{equation*}\delta_i[X^{wW_P}]=\begin{cases} [X^{s_iwW_P}] & \textrm{ if } s_iw < w \/; \\ 0 & \textrm{ otherwise} \/. \end{cases}\end{equation*} Note that if $s_i w < w$ and $w \in W^P$, then automatically $s_iw \in W^P$.




\section{Demazure-Lusztig operators and Chern-Schwartz-MacPherson classes}
In this section we prove that the Chern-Schwartz-MacPherson (CSM) and the Segre-MacPherson (SM) classes of Schubert cells  can be obtained recursively using the left DL operators. This explains the ``R-matrix recursion" from \cite{rimanyi.varchenko:csm}. The version of this statement using the right operators was proved in \cite{aluffi.mihalcea:eqcsm} and \cite{AMSS:shadows}. In particular, this clarifies the fact that in fact there are {\em two} recursions for the CSM classes in $G/B$ (for the left and for the right DL operators). 
\subsection{Chern-Schwartz-MacPherson classes} We start by recalling the definition of the CSM and SM classes. For now let $X$ be any complex algebraic variety. 
The (additive) group $\calF(X)$ of constructible functions consists of finite combinations $\sum c_W \one_W$ where the coefficients $c_W \in \Z$, the (finite) sum is over constructible subsets $W \subset X$, and $\one_W$ is the indicator function. If $f: X \to Y$ is a proper morphism, then there exists a push forward $f_*: \calF(X) \to \calF(Y)$ defined by $f_*(\one_W)(y):=\chi(f^{-1}(y) \cap W)$, where $\chi$ denotes the topological Euler characteristic with compact support.

According to a conjecture attributed to Deligne and Grothendieck, there is a unique transformation 
$c_*: \calF(X) \to H_*(X)$ from the functor $\calF$ of constructible functions on a complex algebraic variety $X$ to 
homology of $X$, which commutes with proper morphisms $f:X \to Y$, and satisfying the normalization property that if $X$ 
is smooth then $c_*(\one_X)=c(TX)\cap [X]$. This conjecture was proved by MacPherson \cite{macpherson:chern}.
~The class $c_*(\one_X)$ for possibly singular $X$ 
was shown to coincide with a class defined earlier by M.-H.~Schwartz \cite{schwartz:1, schwartz:2}. 
For any constructible subset $W\subset X$, the class $\csm(W):=c_*(\one_W)\in H_*(X)$ is the 
\textit{Chern-Schwartz-MacPherson} (CSM) class of $W$ in $X$. If $X$ is smooth, we let 
\[ \ssm(W):=\frac{c_*(\one_W)}{c(TX)} \] denote the \textit{Segre-MacPherson} (SM) class of $W$, 
see \cite{AMSS:ssmpos}. Observe that $c(TX) = 1 + \kappa$ where $\kappa \in H^*(X)$ is nilpotent, thus it makes sense to divide by $c(TX)$.
The theory of CSM classes was extended to the equivariant setting by Ohmoto \cite{ohmoto:eqcsm}. In 
the equivariant case, the SM class needs to be considered in a suitable completion of the equivariant ring $H^*_T(X)$.


\subsection{CSM and SM classes of Schubert cells} We now take $X=G/P$ and consider the equivariant CSM and SM classes of Schubert cells. Part (a) of the following theorem was proved in \cite{aluffi.mihalcea:eqcsm}, and part (b) in \cite{AMSS:shadows}:  
\begin{thm}\label{thm:csm} 
(a) Assume that $X=G/B$, and let $s_i \in W$ be a simple reflection, and $w \in W$. Then
\[ \mathcal{T}_i^R \csm(X_w^\circ) = \csm(X_{ws_i}^\circ).\]
In particular, $\csm(X_w^{\circ})=T_{w^{-1}}^R[X_{id}]$.

(b) For arbitrary $X=G/P$, the SM classes are Poincar{\'e} dual to CSM classes, i.e. for any $u,w\in W^P$,
\[\langle\csm(X_{wW_P}^{\circ}),\ssm(X^{uW_P,\circ})\rangle=\delta_{w,u}.\]
\end{thm} 
One may consider the Schubert expansion of the CSM and SM classes: \[ \csm(X_w^\circ) = \sum a_{w,v} [X_v]; \quad \ssm(X^{w,\circ})= \sum b_{w,v} [X^v] \/. \]
The transition matrices $(a_{w,v})$ and $(b_{w,v})$ have been studied in \cite{aluffi.mihalcea:eqcsm,AMSS:shadows,AMSS:ssmpos}. In the non-equivariant case, it was proved in \cite{AMSS:shadows} that $a_{w,v} \ge 0$ and in \cite{AMSS:ssmpos} that $(-1)^{\ell(w)- \ell(v)} b_{w,v} \ge 0 $. In both situations, and for the equivariant classes, the transition matrices are triangular with respect to the partial Bruhat order, and the coefficients on the diagonal are all non-zero. The latter statement follows from the localization formulae of the classes $\csm(X_w^\circ)$ and $\ssm(X^{w,\circ})$ at the fixed point $w$, proved in \cite[Prop.~6.5]{aluffi.mihalcea:eqcsm}. It follows that both $\{ \csm(X_w^\circ) \} $ and $\{ \ssm(X^{w,\circ}) \}$ are $H^*_T(pt)_{\loc}$-bases over the localized equivariant cohomology ring $H^*_T(G/P)_{\loc}$. Formulas for the structure constants of the multiplication in the CSM or SM basis of Schubert cells can be found in \cite{su2019structure}, generalizing the results from \cite{goldin.knutson:schubert}. From this we obtain the following result, implicit in \cite{AMSS:shadows}.

\begin{cor}\label{cor:Tissm} Assume that $X=G/B$, let $s_i \in W$ be a simple reflection, and let $w \in W$. Then \[ \mathcal{T}_i^{R,\vee} \ssm(X^{w,\circ}) = \ssm(X^{ws_i, \circ})\quad \textrm{ and } \quad \mathcal{T}_i^{R,\vee} \ssm(X_{w}^{\circ}) = \ssm(X_{ws_i}^{\circ}) \/.\]
\end{cor}
\begin{proof} Using Theorem \ref{thm:csm} and equation \eqref{E:radj} we obtain that for any $u \in W$,
\[ \begin{split} \langle \csm(X_u^\circ), \mathcal{T}_i^{R,\vee} \ssm(X^{w,\circ}) \rangle & = \langle \mathcal{T}_i^R \csm(X_u^\circ), \ssm(X^{w,\circ}) \rangle = \langle \csm(X_{us_i}^\circ), \ssm(X^{w,\circ}) \rangle \\ & = \langle \csm(X_{u}^\circ), \ssm(X^{ws_i,\circ}) \rangle \/. \end{split} \] Then the first identity follows from the fact that the Poincar{\'e} pairing is non-degenerate. The second is proved in a similar way.\end{proof}

The main result of this section is the following theorem.
\begin{thm}\label{thm:csmleft} Let $X=G/P$, let $s_i \in W$ be a simple reflection, and let $w \in W$. Then the following hold:
\[ \mathcal{T}_i^L \csm(X_{wW_P}^{\circ}) = \csm(X_{s_iwW_P}^{\circ}) \/; \quad  \mathcal{T}_i^{L,\vee} \ssm(X^{wW_P,\circ}) = \ssm(X^{s_iw W_P,\circ}),\]
and 
\[ \mathcal{T}_i^{L,\vee}\csm(X^{wW_P,\circ}) = \csm(X^{s_iw W_P,\circ}) \/; \quad \mathcal{T}_i^L\ssm(X_{wW_P}^{\circ}) = \ssm(X_{s_iwW_P}^{\circ}) \/.\]
\end{thm}


\begin{proof} Since $X_w^{\circ} = w_0. X^{w_0 w W_P, \circ}$, and by functoriality of CSM classes, it follows that $w_0^L\csm(X_w^{P,\circ})=\csm(X^{w_0w W_P, \circ})$. Further, $w_0^L\ssm(X_w^{\circ})=\ssm(X^{w_0wW_P,\circ})$, because $w_0^L c(T_{G/P}) = c(T_{G/P})$, since the tangent bundle $T_{G/P}$ is $G$-equivariant. From this and the fact that $\mathcal{T}_{\alpha_i}^{L,\vee} = w_0^L \mathcal{T}_{-w_0\alpha_i}^L w_0^L$ (by Lemma \ref{lemma:DLcomm}(c)), we deduce that the second row of equalities follows from those in the first row.

We now prove the first identity in the first row. First consider the situation when $P=B$. If $w=id$, then 
\[ \begin{split} \mathcal{T}_i^L[X_{id}] & = (-\delta_i + s_i^L)[X_{id}] 
\\ &= (1+\alpha_i)[X_{s_i}]+[X_{id}] =\csm(X_{s_i}^\circ) \/,\end{split} \] 
where the second equality follows from the equation \eqref{E:delschub} and Corollary \ref{cor:action}. For the general case, we apply Lemma \ref{lemma:DLcomm}, to obtain
\[ \begin{split} \mathcal{T}_i^L \csm(X_w^\circ) & = \mathcal{T}_i^L T_{w^{-1}}^R [X_{id}] = T_{w^{-1}}^R \mathcal{T}_i^L [X_{id}] \\ & = T_{w^{-1}}^R \csm(X_{s_i}^\circ) = \csm(X_{s_i w}^\circ)\/. \end{split} \] 
We now allow $P$ to be arbitrary, and recall that $\pi: G/B \to G/P$ is the projection. By functoriality of CSM classes $\pi_* \csm(X_w^{\circ}) = \csm(X_{wW_P}^{\circ})$ {(see e.g.~\cite[Prop.~3.5]{aluffi.mihalcea:eqcsm})}. Since $\pi_*$ and $\mathcal{T}_i^L$ commute, 
\begin{align*}
\mathcal{T}_i^L(\csm(X_{wW_P}^\circ))=&\mathcal{T}_i^L\pi_*\csm(X_{w}^\circ)=\pi_*\mathcal{T}_i^L\csm(X_{w}^\circ)=\pi_*\csm(X_{s_iw}^\circ)=\csm(X_{s_iwW_P}^\circ).
\end{align*}
This finishes the proof of the first identity. The second identity follows by Poincar{\'e} duality, as in the proof of the Corollary \ref{cor:Tissm}, utilizing Lemma \ref{lemma:DLcomm}(b) and that
 \begin{align*}
\langle \csm(X_{uW_P}^\circ),\mathcal{T}_i^{L,\vee}\ssm(X^{wW_P,\circ})\rangle=&s_i.\langle \mathcal{T}_i^{L}\csm(X_{uW_P}^\circ),\ssm(X^{wW_P,\circ})\rangle\\
=&s_i.\langle \csm(X_{s_iuW_P}^\circ),\ssm(X^{wW_P,\circ})\rangle\\
=&\delta_{s_iu W_P,w W_P}.
\end{align*}
\end{proof}

\subsection{R-matrix recursions}\label{sec:Rmat} We show next that the left $W$-action on CSM classes recovers the R-matrix recursions for the stable envelopes from \cite{maulik2019quantum,RTV:partial}. 

Let $X:=G/P$ and denote by $\iota:X \to T^*_{X}$ the zero section. The action of $T$ on $X$ induces an action $T^*_X$, and in addition there is an action of $\C^*$ on $T^*_X$ given by dilation with character 
$\hbar^{-1}$. The pull-back 
$\iota^*: H_{2 \dim X}^{T \times \C^*}(T^*_X) \to H_{0}^{T \times \C^*}(X)$ is a $H^*_{T \times \C^*}(pt)$-linear isomorphism. 
(Here we pass from cohomology to homology by Poincar{\'e} duality.) 
Observe that since $\C^*$ acts trivially on $X$, every class
$\kappa \in H_0^{T \times \C^*}(X)$ can be written uniquely as 
$\kappa = \sum a_i \hbar^i$ where $a_i \in H_{2 i}^T(X)$. 
It follows that we may identify $H_0^{T \times \C^*}(X) \simeq H_*^T(X)$ with 
the correspondence given by homogenization by $\hbar$. See \cite{AMSS:shadows} for more details on this homogenization.

Maulik and Okounkov defined in \cite{maulik2019quantum} the {\em stable envelopes}; see also \cite{su:restriction}. A stable envelope $\mathrm{stab}_{\mathcal{C}_\sigma}(wW_P)$ is an element in $H_{2 \dim X}^{T \times \C^*}(T^*_X)$ defined by certain interpolation conditions. It depends on a chamber $\mathcal{C}_\sigma \subset Lie(T)$ for $\sigma \in W$, and on the fixed point $e_{w W_P}$. By convention
$\mathcal{C}_{id}$ is the dominant chamber, and $\mathcal{C}_{\sigma} = \sigma . \mathcal{C}_{id}$. Since $\iota^*$ is an isomorphism such that $w^L \iota^* = \iota^* w^L$ ($w \in W$), the following result shows that the stable envelopes are determined by the CSM classes.

\begin{thm}\label{thm:stabcsm}\cite{rimanyi.varchenko:csm,AMSS:shadows,AMSS:motivic} Let $\sigma,w \in W$, and $s_i$ a simple reflection. 

(a) The left action of $W$ on $\mathrm{stab}_{\mathcal{C}_\sigma}(wW_P)$ is given by \[ s_i^L. \mathrm{stab}_{\mathcal{C}_\sigma}(wW_P) =  \mathrm{stab}_{\mathcal{C}_{s_i \sigma}}(s_i wW_P) \/. \]

(b) The CSM classes are equivalent to the stable envelopes for the identity chamber: 
\[ \iota^* \mathrm{stab}_{\mathcal{C}_{id}}(wW_P) = (-1)^{\dim X} \csm^{\hbar}(X_{wW_P}^\circ) \/, \] where 
$\csm^{\hbar}(X_{wW_P}^\circ)$ denotes the homogenization of the ordinary CSM class.
\end{thm}
\begin{proof} Part (a) is proved (more generally for the K-theoretic stable envelopes) in \cite[Lemma 8.2]{AMSS:motivic}. Part (b) is \cite[Cor.~6.6]{AMSS:shadows}; see also \cite{rimanyi.varchenko:csm}.\end{proof}

By Theorem \ref{thm:csmleft}, the homogenized CSM classes may be calculated recursively from the homogenized left DL operator: \[ \mathcal{T}_{i}^{L,\hbar} := s_i^L- \hbar \delta_i  = (1+ \frac{\hbar}{\alpha_i}) s_i^L - \frac{\hbar}{\alpha_i} id  \/.\] Solving for $s_i^L$ we obtain \[ s_i^L = \frac{\alpha_i}{\alpha_i + \hbar} \mathcal{T}_i^{L,\hbar} + \frac{\hbar}{\alpha_i + \hbar} id \/. \] (This may also be obtained from the R-matrix with spectral parameter $R(u)$ \cite[eq.~(4.1) and p.~136]{maulik2019quantum}, after making the substitutions $\mathbf{s} \mapsto \mathcal{T}_i^{L,\hbar}$, $u \mapsto \alpha_i$ and then homogenizing with the character $-\hbar$.)
The left $W$ action on homogenized CSM classes follows immediately from the previous expression and 
Theorem \ref{thm:csmleft}.
\begin{prop} Let $s_i$ be a simple reflection, and $w \in W^P$. Then 
\[ s_i^L . \csm^\hbar(X_{wW_P}^\circ) = 
\frac{\hbar}{\hbar+\alpha_i}\csm^\hbar(X_{wW_P}^{\circ})+\frac{\alpha_i}{\hbar+\alpha_i}\csm^\hbar(X_{s_iwW_P}^{\circ}) \/.\]
\end{prop}
Combined with Theorem \ref{thm:stabcsm}, this Proposition recovers the recursion for the weight functions from \cite[Lemma 3.6]{RTV:partial}; by \cite[Thm.~4.1]{RTV:partial} (the main theorem in {\em loc.~cit.}) that recursion is the same as the recursion for the stable envelopes.
 

\section{Left divided difference operators in equivariant K-theory} Our next goal is to prove a K theoretic analogue of Theorem \ref{thm:csmleft}.
 We start by recalling some basics on equivariant K-theory, and then we will define the K-theoretic versions of BGG and DL operators.

\subsection{Preliminaries} Let $X$ be a smooth projective variety with an action of a linear algebraic group $G$. For an introduction to equivariant K-theory, and more details, see \cite{chriss2009representation}. The equivariant K-theory ring $K_G(X)$ is the Grothendieck ring generated by symbols $[E]$, where $E \to X$ is an $G$-equivariant vector bundle, modulo the relations $[E]=[E_1]+[E_2]$ for any short exact sequence $0 \to E_1 \to E \to E_2 \to 0$ of equivariant vector bundles. The additive ring structure is given by direct sum, and the multiplication is given by tensor products of vector bundles. Since $X$ is smooth, any coherent sheaf has a finite resolution by vector bundles, and the ring $K_G(X)$ coincides with the Grothendieck group of $G$-linearized coherent sheaves on $X$. 

The ring $K_G(X)$ is an algebra over $K_G(pt) = R(G)$, the representation ring of $G$. If $G=T$ is a torus, then this is the Laurent polynomial ring $K_T(pt) =\Z[e^{\pm t_1}, \ldots , e^{\pm t_r}]$ where $e^{\pm t_i}$ are characters corresponding to a basis of the Lie algebra of $T$ (and $t_i$ correspond to the generators for $H^*_T(pt)$ from above). Since $X$ is proper, the push-forward to a point equals the Euler characteristic, or, equivalently, the virtual representation, \[\chi(X, \mathcal{F})= \int_X [\mathcal{F}] := \sum_i (-1)^i H^i(X, \mathcal{F}) \/. \] 
In particular, for $E,F$ equivariant vector bundles, this gives a pairing \[ \langle - , - \rangle :K_G(X) \otimes K_G(X) \to K_G(pt); \quad \langle [E], [F] \rangle := \int_X E \otimes F = \chi(X, E \otimes F) \/. \] To simplify notation, we will denote by the same symbol a vector bundle $E$ and its K-theory class. 

Any $G$-linearized coherent sheaf $\mathcal{F}$ on $X$ determines a class $[\mathcal{F}] \in K_G(X)$. In particular, if $\Omega \subset X$ is a $G$-stable subscheme, then its structure sheaf determines a class $[\cO_\Omega] \in K_G(X)$.~If $\Omega = \{ x \}$ is a point, we will denote its (equivariant) K-theory class by $\iota_x$.

Assume now that $G=T$ is a complex torus, and let $V$ be a (complex) vector space with a $T$-action, with weight decomposition 
$V = \oplus_i V_{\mu_i}$, where each $\mu_i$ is a weight in the dual of the Lie algebra of $T$. 
The {\em character} of $V$ is the element $ch(V):=\sum_i \dim V_{\mu_i} e^{\mu_i}$, regarded  in $K_T(pt)$. If $y$ is an indeterminate, 
the $\lambda_y$ class of $V$, denoted $\lambda_y(V)$, is the element 
\[ \lambda_y(V) =\sum_{i \ge 0} y^i ch(\wedge^i V) \in K_T(pt)[y] \/.\] 
The $\lambda_y$ class is multiplicative for short exact sequences, i.e. if $0 \to V_1 \to V_2 \to V_3 \to 0$ 
is a short exact sequence then $\lambda_y(V_2) = \lambda_y(V_1) \lambda_y(V_3)$. In particular, 
$\lambda_y(V) = \prod_i (1+y e^{\mu_i})^{\dim V_{\mu_i}}$; see \cite{hirzebruch:topological}.

We recall next a version of the localization theorem in the case when $G=T$ is a complex torus, which will be used throughout this note. Our main reference is Nielsen's paper \cite{nielsen:diag}; see also \cite{chriss2009representation}. 

Let $S$ be the subset of $K_T(pt)$ generated by elements of the form $1-e^\mu$ for nontrivial torus weights $\mu$. Then $S \subset R(T)$ is a multiplicative subset and $0 \notin S$. If $V$ is a $T$-module such that the fixed locus $V^T = \{ 0 \}$, then $S$ contains the element $\lambda_{-1}(V) = \prod (1-e^{\mu_i})^{\dim V_{\mu_i}}$. Denote by $K_T(X)_{\loc}$ respectively $K_T(pt)_{\loc}$ the localization of $K_T(X)$ and of $K_T(pt)$ at $S$. Since $R(T)$ is a domain, we may identify $K_T(X)$ with a subring inside its localization. For each $x \in X^T$, let $i_x: \{ x \} \to X$ denote the inclusion. This is a $T$-equivariant proper morphism, and it induces a map $i_x^*: K_T(X) \to K_T(\{x\}) = K_T(pt)$. We need the following simplified version of the localization theorem; cf.~ \cite{nielsen:diag}.

\begin{thm}\label{thm:loc} Assume that the fixed point set $X^T$ is finite, and let $N$ be the normal bundle of $X^T$ in $X$. Then the following hold:

(a) The class $\lambda_{-1}(N^\vee)$ is a unit in $K_T(X)_{\loc}$;

(b) When $x$ varies in $X^T$, the structure sheaves $\iota_x:=[\cO_x]$ of the fixed points form a $K_T(pt)_{\loc}$-basis of $K_T(X)_{\loc}$;

(c) For any $T$-linearized coherent sheaf $\mathcal{F}$ on $X$, the following formula holds:

\[ [\mathcal{F}] = \sum_{x \in X^T} \frac{i_x^* [\mathcal{F}]}{\lambda_{-1}(T^*_x X)} \iota_x \in K_T(pt)_{\loc} \/. \]

\end{thm}

We now specialize to the case when $X:=G/P$, with the $G$-action given by left multiplication, and the torus $T$ acting by restriction. The Schubert varieties $X_w$ and $X^w$ determine classes $\calO_w:= [\cO_{X_w}]$ and $\calO^w:= [\cO_{X^w}]$ in $K_T(X)$. Similarly, the $T$-fixed points give classes denoted by $\iota_w:= [\cO_{e_w}]$. The equivariant K-theory $K_T(G/P)$ is a free module over $K_T(pt)$ with a basis given by Schubert classes $\{\cO_w \}_{w \in W^P}$ respectively $\{\cO^w \}_{w \in W^P}$. By the localization theorem \ref{thm:loc}, $\{\iota_w|w\in W\}$ is a basis for the localized ring $K_{T}(X)_{\loc}$. For any $\gamma\in K_T(G/P)$, let $\gamma|_w\in K_T(pt)$ denote the pullback of $\gamma$ to the fixed point $e_w$. If $P=B$, and $\lambda$ is any torus weight, let $\calL_\lambda:=G\times^B\bbC_\lambda$ denote the corresponding line bundle. We also record the following well-known lemma. 

\begin{lemma}\label{lemma:proj} Let $\pi:G/B \to G/P$ the projection and let 
$\Omega \subset G/B$ and $\Omega' \subset G/P$ be any ($B$ or $B^-$-stable) Schubert varieties. Then 
\[ \pi_*[\cO_\Omega] = [\cO_{\pi(\Omega)}] \in K_T(G/P) \/; \quad  \pi^*[\cO_{\Omega'}] = [\cO_{\pi^{-1}(\Omega')}] \in K_T(G/B) \/. \]
\end{lemma}
\begin{proof} The first equality follows from \cite[Thm. 3.3.4(a)]{brion.kumar:frobenius} and the second because $\pi$ is a flat morphism.\end{proof}

\subsection{Left and right Demazure operators}\label{sec:KDem} Let $X=G/B$. The left and right multiplications by Weyl group elements defined in \S \ref{sec:twoWactions} induce left and right $W$-actions on $K_T(X)$. We use the same notation of $w^L$ and $w^R$ respectively as in cohomology. In terms of localization, we have
\[w^L(\gamma)|_u=w(\gamma|_{w^{-1}u}),\quad \textit{and}\quad w^R(\gamma)|_u=\gamma|_{uw},\]
where $\gamma\in K_T(X)$ and $u,w\in W$. As in cohomology, the left action is well defined for $K_T(G/P)$ where $P$ is an arbitrary parabolic. 

\begin{prop}\label{prop:Kprop} Let $w \in W$ and $P$ be an arbitrary parabolic subgroup. The left and right $W$-actions $w^L$ and $w^R$ have the following properties:

(a) $w^R$ is a $K_T(pt)$-linear algebra automorphisms of $K_T(G/B)$;

(b) $w^L$ is a ring automorphism of $K_T(G/P)$. In particular, for any $\lambda \in Lie(T)^*$ and $a \in K_T(G/P)$, \[ w^L(e^\lambda \cdot a) = e^{w(\lambda)} \cdot w^L(a) \/. \] 

(c) For any $u,v \in W$, $u^L v^R = v^R u^L$ as automorphisms of $K_T(G/B)$.

(d) Let $Q \subset P $ be any parabolic subgroups and $\pi:G/Q \to G/P$ the projection. Then the left action $w^L$ commutes with $\pi_*: K_T(G/Q) \to K_T(G/P)$ and with $\pi^*: K_T(G/P) \to K_T(G/Q)$. 

(e) For any $\gamma_1,\gamma_2 \in K_T(G/P)$, and $w\in W$, 
\[\langle w^L(\gamma_1),w^L(\gamma_2)\rangle=w.\langle\gamma_1,\gamma_2\rangle \/.\]

(f) For any simple reflection $s_i\in W$, \[ s_i^L. \iota_{wW_P} = \iota_{s_i wW_P} \] as classes in $K^*_T(G/P)_{\loc}$.

(g) For any simple reflection $s_i\in W$, \[ s_i^R .\iota_w = - e^{w(\alpha_i)} \iota_{ws_i} \] as classes in $K^*_T(G/B)_{\loc}$.

\end{prop}

\begin{proof} Parts (a)-(e) have the same proofs as for equivariant cohomology, based on analyzing the left and right multiplication morphisms. Part (f) follows because $s_i^L$ is induced from the left multiplication by $s_i$. Part (g) follows from a localization argument, as follows. 
Recall that the fixed point classes are determined by the localization formulae $(\iota_w)|_u= \delta_{w,u}\prod_{\alpha>0} (1-e^{w(\alpha)})$. Then \[ \frac{(s_i^R.\iota_w)|_{ws_i}}{\iota_{ws_i}|_{ws_i}} = \frac{\iota_w|_w}{\iota_{ws_i}|_{ws_i}} = \frac{\prod_{\alpha >0} (1- e^{w(\alpha)})}{\prod_{\alpha >0} (1 -e^{ws_i(\alpha)})} = w \cdot \frac{1-e^{\alpha_i}}{1-e^{-\alpha_i}} =-e^{w(\alpha_i)} \/. \] This finishes the proof of part (g).\end{proof} 

We move on to the definition of the left Demazure operators $\delta_i:=\delta_{\alpha_i}$ and the right Demazure operators $\partial_i:=\partial_{\alpha_i}$, where $\alpha_i$ is a simple root. 

The right operator has the same definition as in cohomology: $\partial_i = \pi_i^* (\pi_i)_*$, where $\pi_i:G/B \to G/P_i$ is the projection, and $P_i$ is the minimal parabolic subgroup. This operator appeared in \cite{demazure:desingularisations} in relation to the study of K-theory of flag manifolds. In terms of localization, we have
\[\partial_i(\gamma)|_v=\frac{\gamma|_v-e^{v\alpha_i}s_i^R(\gamma)|_v}{1-e^{v\alpha_i}},\]
where $\gamma\in K_T(G/B)$ and $v\in W$. The Demazure operators satisfy $\partial_i^2=\partial_i$ and the braid relations. The definition of $\partial_i$ and the Lemma \ref{lemma:proj} imply that \begin{equation}\label{E:demschub} \partial_i (\cO_w) = \begin{cases} \cO_{ws_i} & \textrm{if } ws_i > w \/; \\ \cO_w & \textrm{otherwise} \/; \end{cases} \quad 
\partial_i (\cO^{w}) = \begin{cases} \cO^{ws_i} & \textrm{ if } ws_i < w \/; \\ \cO_w & \textrm{otherwise} \/. \end{cases} \end{equation} The braid relations imply that $\partial_w$ is well defined and by equation \eqref{E:demschub}, $\calO_w=\partial_{w^{-1}}(\calO_{id})$.

Define two variants of left Demazure operators
\begin{equation}\label{E:KleftDem} \delta_i := \frac{1}{1-e^{\alpha_i}}(1-e^{\alpha_i}s_i^L),\quad 
\delta_i^\vee:=\frac{1}{1-e^{-\alpha_i}}(1-e^{-\alpha_i}s_i^L).\end{equation} Using the left $W$-action, this definition gives (module) endomorphisms of $K_T(G/P)$. The two variants are related by: 
\[w_0^L\delta_{\alpha_i} w_0^L=\delta_{-w_0(\alpha_i)}^\vee \/.\]
{The operators $\delta_i$ and $\delta_i^\vee$ are {\em a priori} defined only in $K_T(G/P)_{\loc}$, but in fact they preserve the non-localized ring.~To see this, observe that the Schubert classes $\cO_w$ form a $K_T(pt)$-basis of $K_T(G/P)$, then utilize this together with the Leibniz rule and Proposition \ref{prop:Kaction}(b).} An algebraic version of the operator $\delta_i^\vee$ appeared in \cite[Eq.~($I_4$)]{kostant.kumar:KT}, and it is proved in \textit{loc. cit.} that they satisfy $(\delta_i^\vee)^2=\delta_i^\vee$ and the braid relations. From that it follows that the geometric operators satisfy the same relations, and the same is true for $\delta_i$. The operators $\delta_i$ have been defined for any variety $X$ with a $G$ action in \cite{harada.landweber.sjamaar:divided} and may also be defined using certain convolutions; see the Appendix below. In fact, Proposition \ref{prop:eqdelta} provides another proof that $\delta_i$ preserves $K_T(G/P)$. {If $X=G/B$, there is a particularly vivid symmetry between the left and right Demazure operators, obtained by viewing these operators either as a left/right convolution, or as a single operator acting on first/second factor of $K_T(G/B) \simeq R(T) \otimes_{R(G)} R(T)$; see \S \ref{appex:two} below.}

Next we record some basic properties of these operators.
\begin{lemma}\label{lemma:Kalg} (a) The operator $\partial_i$ is a $K_T(pt)$-module homomorphism, i.e. for any $e^\lambda \in K_T(pt)$ and $\eta \in K_T(G/B)$, $\partial_i(e^\lambda \eta) = e^\lambda \partial_i(\eta)$.

(b) The operator $\delta_i$ is a $K_G(G/P)$-module homomorphism, i.e. for any $\kappa \in K_G(G/P)$ and $\eta \in K_T(G/P)$, $\delta_i(\kappa \eta) = \kappa \delta_i(\eta)$.

(c) For any $i,j$:
\[ \delta_i s_j^R = s_j^R \delta_i \/; \quad \partial_i s_j^L = s_j^L \partial_i \/; \quad \delta_i \partial_j = \partial_j \delta_i \/. \]

(d) (Leibniz rule) For any $a, b \in K_T(X)$, \[ \delta_i(a \cdot b) = \delta_i(a) \cdot b + e^{\alpha_i} s_i^L(a) \cdot \delta_i(b) - e^{\alpha_i} s_i^L(a) \cdot s_i^L(b) \/. \] 
\end{lemma}
\begin{proof} The proofs for (a),(b), (c) are the same as the ones from cohomology, using now Proposition \ref{prop:Kprop}. Part (d) is an easy calculation.
\end{proof}  
We now turn to the action of the left operators on Schubert classes.
\begin{prop}\label{prop:Kaction} (a) 
For any simple root $\alpha_i$ and $w\in W$, the following holds in $K_T(G/B)$:
\[ s_i^L (\cO_w)= \begin{cases} e^{-\alpha_i}\cO_w+(1-e^{-\alpha_i})\cO_{s_i w} & \textrm{if } s_i w> w\/; \\ \cO_w & \textrm{otherwise} \/. \end{cases} \]

(b) Let $P$ any parabolic subgroup. Then in $K_T(G/P)$,  
\[\delta_i (\cO_{wW_P})= \begin{cases} \cO_{s_i w W_P} & \textrm{if } s_i w> w\/; \\ \cO_{w W_P} & \textrm{otherwise} \/. \end{cases} \]
\end{prop}
\begin{remark}\label{rmk:dualDem}
Conjugating by $w_0^L$, we obtain similar result for the opposite Schubert classes: \[\delta_i^\vee (\cO^{wW_P})= \begin{cases} \cO^{s_i w W_P} & \textrm{if } s_i w< w\/; \\ \cO^{w W_P} & \textrm{otherwise} \/. \end{cases} \]
\end{remark}
\begin{proof}[Proof of Proposition \ref{prop:Kaction}] We utilize that $s_i^L\partial_w=\partial_w s_i^L$ and $\cO_w=\partial_{w^{-1}}(\cO_{id})$ to obtain that 
\[ s_i^L (\cO_w)=s_i^L\partial_{w^{-1}}(\cO_{id}) =\partial_{w^{-1}}s_i^L(\cO_{id})  =\partial_{w^{-1}}(\iota_{s_i}) \/. \] 
One may use localization to prove that $\iota_{s_i}=(1-e^{-\alpha_i})\cO_{s_i}+e^{-\alpha_i}\cO_{id}$, and from this we deduce 
\[ \partial_{w^{-1}}(\iota_{s_i}) =  \begin{cases} e^{-\alpha_i}\cO_w+(1-e^{-\alpha_i})\cO_{s_i w} & \textrm{if } s_i w> w\/; \\ \cO_w & \textrm{otherwise} \/. \end{cases} \/. \]
This proves (a). If $P=B$, part (b) follows from part (a) and the definition of $\delta_i$. For arbitrary parabolic $P$, by Lemma \ref{lemma:proj},
$\pi_* (\cO_w)=\cO_{wW_P}$, thus 
\[ \delta_i(\cO_{wW_P}) = \delta_i ( \pi_* (\cO_w)) = \pi_* \delta_i (\cO_w) \/. \] The result follows from the case when $P=B$.\end{proof}

\begin{remark} For another definition of the left and right Demazure operators on $G/B$ using convolutions, see Proposition \ref{prop:convos} below. 
\end{remark}

\subsection{Left and right Demazure-Lusztig operators}\label{ss:leftDL} Recall the definition of the {\em right} Demazure-Lusztig (DL) operators \cite{AMSS:motivic} 
\[ \mathcal{T}_i^R = (1+ y \mathcal{L}_{\alpha_i}) \partial_i -id; \quad \mathcal{T}_i^{R,\vee} = \partial_i(1+y \mathcal{L}_{\alpha_i}) -id \/.\] 
They satisfy the braid relation and the quadratic relation 
\[(\mathcal{T}^R_i+1)(\mathcal{T}^R_i+y)=(\mathcal{T}^{R,\vee}_i+1)(\mathcal{T}^{R,\vee}_i+y)=0.\]
The operator $\mathcal{T}_i^{R, \vee}$ appeared in \cite[Eq. (4.2)]{lusztig:eqK}, and $\mathcal{T}_i^R$ in \cite{lee.lenart.liu:whittaker,brubaker.bump.licata}.

The {\em left} DL operators are defined in the following way
\[ \mathcal{T}_i^L: = \delta_i (1+ye^{\alpha_i}) - id=\frac{1+ye^{-\alpha_i}}{1-e^{-\alpha_i}}s_i^L- \frac{1+y}{1-e^{-\alpha_i}}\/; \]
\[\mathcal{T}_i^{L,\vee}:=\delta_i^\vee(1+ye^{-\alpha_i})-id=\frac{1+ye^{\alpha_i}}{1-e^{\alpha_i}}s_i^L- \frac{1+y}{1-e^{\alpha_i}}.\]
\noindent Observe that $w_0^L\mathcal{T}_{\alpha_i}^L w_0^L=\mathcal{T}_{-w_0\alpha_i}^{L,\vee}$. Both the operators $\mathcal{T}_i^L$ and 
$\mathcal{T}_i^{L,\vee}$ satisfy the quadratic relation $(\mathcal{T}_i^{L}+1)(\mathcal{T}_i^{L}+y)=0$ (and same for $\mathcal{T}_i^{L,\vee}$) and the braid relations. In particular, 
they are invertible. 

\begin{remark} Under the identification by the Atiyah-Borel isomorphism $K_T(pt) \simeq K_G(G/B)$ defined by 
$e^\lambda \mapsto \mathcal{L}_{-\lambda} = G \times^B \C_{-\lambda}$, the operator 
$\mathcal{T}_i^{L,\vee}$ appeared in \cite[Equation (8.1)]{lusztig:eqK}, where the quadratic and braid relations are proved.
For partial flag manifolds in Lie type A, the operator $\mathcal{T}_i^L$ also appeared in \cite[eq.~(9.1)]{RTV:Kstable}.\end{remark}

As usual, the left operators are well defined as endomorphisms of $K_T(G/P)$, where $P$ is an arbitrary parabolic group. 

\begin{lemma}\label{lemma:Kcomm} The following properties hold.

(a) For any $i,j$, and as operators in $K_T(G/B)$: \[ s_j^L \mathcal{T}^R_i = \mathcal{T}^R_i s_j^L \/; \quad s_j^L \mathcal{T}^{R,\vee}_i = \mathcal{T}^{R,\vee}_i s_j^L \/; \quad \mathcal{T}^R_i T^L_j = \mathcal{T}^L_j \mathcal{T}^R_i,\quad \mathcal{T}^{R,\vee}_i T^{L,\vee}_j = \mathcal{T}^{L,\vee}_j \mathcal{T}^{R,\vee}_i;\]

(b) If $\kappa \in K_G(G/P)$ and $\gamma \in K_T(G/P)$, then \[ \mathcal{T}_i^L(\kappa \cdot \gamma) = \kappa \cdot \mathcal{T}_i^L(\gamma) \/; \quad \mathcal{T}_i^{L,\vee}(\kappa \cdot \gamma) = \kappa \cdot \mathcal{T}_i^{L,\vee}(\gamma) \]

(c) For any $\gamma_1,\gamma_2\in K_T(G/B)$, and any $\gamma_3, \gamma_4 \in K_T(G/P)$, 
\[ \langle \mathcal{T}_i^{R} (\gamma_1),\gamma_2\rangle= \langle\gamma_1,\mathcal{T}_i^{R,\vee} (\gamma_2)\rangle \/; \quad \langle \mathcal{T}_i^{L} (\gamma_3),\gamma_4\rangle=s_i . \langle\gamma_3,\mathcal{T}_i^{L,\vee} (\gamma_4)\rangle \/.\]

(d) (Leibniz formula) {For any $a,b \in K_T(G/P)$, \[ \mathcal{T}_i^L(a\cdot b) = \mathcal{T}_i^L(a) \cdot b + s_i^L(a) \cdot \mathcal{T}_i^L(b) + y s_i^L(a) \cdot b \/. \]}
\end{lemma}
\begin{proof} Parts (a) and (b) follow from the previous commutativity properties, using that the line bundle class $\mathcal{L}_{\alpha_i}$ is $G$-equivariant, thus its class is fixed under the left $W$-action. The first equality in (c) was proved in \cite[Lemma 3.3]{AMSS:motivic}, and the second equality and part (d) are straightforward calculations.\end{proof}

\section{Motivic Chern classes}
The K-theoretical generalization of the CSM classes are the motivic Chern classes defined by Brasselet, Sch{\"u}rmann and 
Yokura \cite{brasselet.schurmann.yokura:hirzebruch}. In this section we recall their definition and prove some basic properties 
about pull backs and push forwards.
\subsection{Definition} Let $X$ be a quasi-projective complex $T$-variety. The (equivariant) Grothendieck motivic group ${G}_0^T(var/X)$ is the free abelian group generated by classes $[f: Z \to X]$ where $Z$ is a quasi-projective $T$-variety and $f: Z \to X$ is a $T$-equivariant morphism, modulo the usual additivity relations 
\[[f: Z \to X] = [f: U \to X] + [f:Z \setminus U \to X]\] for $U \subset Z$ an open $T$-invariant subvariety.
If $X=pt$ then ${G}_0^T(var/pt)$ is a ring with the product given by the external product of morphisms, and the groups ${G}_0^T(var/X)$ also acquire by the external product a module structure over $G_0^T(var/pt)$. 

For any equivariant morphism $f:X \to Y$ of quasi-projective $T$-varieties there are well defined push-forwards $f_!: G_0^T(var/X) \to G_0^T(var/Y)$ (given by composition) and pull-backs $f^*:G_0^T(var/Y) \to G_0^T(var/X)$ (given by fiber product); see \cite[\S 6]{bittner:universal}.
The following theorem was proved by Brasselet, Sch{\"u}rmann and Yokura \cite[Thm. 2.1]{brasselet.schurmann.yokura:hirzebruch} in the non-equivariant case. The changes required to prove the equivariant case were addressed in \cite{feher2021motivic} and \cite{AMSS:motivic}. 

\begin{thm}\label{thm:existence}\cite[Theorem 4.2]{AMSS:motivic} Let $X$ be a quasi-projective, non-singular, complex algebraic variety with an action of the torus $T$. There exists a unique natural transformation $MC_y: G_0^T(var/X) \to K_T(X)[y]$ satisfying the following properties:
\begin{enumerate} \item[(1)] It is functorial with respect to $T$-equivariant proper morphisms of non-singular, quasi-projective varieties. 
\item[(2)] It satisfies the normalization condition \[ MC_y[id_X: X \to X] = \lambda_y(T^*_X) = \sum y^i [\wedge^i T^*_X]_T \in K_T(X)[y] \/. \]
\end{enumerate}
Further, the transformation $MC_y$ satisfies the following Verdier-Riemann-Roch (VRR) formula.
For any smooth, $T$-equivariant morphism $\pi: X \to Y$ of quasi-projective and non-singular 
algebraic varieties, and any $[f: Z \to Y] \in G_0^T(var/Y)$,
\[ \lambda_y(T^*_\pi) \cdot \pi^* MC_y[f:Z \to Y] = MC_y[\pi^* f:Z \times_Y X \to X] \/. \]
\end{thm}

If one forgets the $T$-action, then the equivariant motivic Chern class above recovers the non-equivariant motivic Chern class from
\cite{brasselet.schurmann.yokura:hirzebruch} (either by its construction, or by the properties (1)-(2) from Theorem~\ref{thm:existence}
and the corresponding results from \cite{brasselet.schurmann.yokura:hirzebruch}.) Most of the time the variety $X$ will be understood from the context. If $Y \subset X$ is a subvariety, not necessarily closed, denote by \[ MC_y(Y) := MC_y[Y \hookrightarrow X] \/. \] If $i: Y \subset X$ is closed submanifold and $Y' \subset Y$ then by functoriality 
$MC_y[Y' \hookrightarrow X] = i_* MC_y [Y' \hookrightarrow Y]$ (K-theoretic push-forward). 

\subsection{Pull backs and motivic Segre classes}
Let $X$ be a quasiprojective complex manifold. Recall the Serre duality functor $\calD(-):=\RHom(-,\omega_{X}^\bullet)$ on $K_T(X)$, where 
$\omega_{X}^\bullet=[\wedge^{\dim X}T^*_X][\dim X]$ is the canonical complex; we extend $\calD$ to $K_T(X)[y^{\pm1}]$ by $\calD(y^i)=y^{-i}$. 
\begin{defn}\label{defn:smc} Let $\Omega \subset X$ be a pure dimensional $T$-stable subvariety of $X$. 
The {\em Segre motivic class} $SMC_y(\Omega) \in K_T(X)$ is the class 
\[ SMC_y(\Omega) := (-y)^{\dim \Omega} \frac{\caD(MC_y(\Omega))}{\lambda_y(T^*_X)}\in K_T(X)[[y^{\pm 1}]]  \/. \] 
(This class lives in an appropriate completion of $K_T(X)[[y^{\pm 1}]]$. The normalization factor $(-y)^{\dim \Omega}$ will simplify the formulae below for pull-back and Poincar{\'e} duality.)
\end{defn}
We prove next a result about pull-backs of motivic Segre classes, to be used later.
\begin{prop}\label{prop:gensegrepb} Let $\pi: X \to Y$ be a smooth equivariant morphism of pure-dimensional quasi-projective manifolds, and let $\Omega \subset Y$ be a $T$-stable pure-dimensional subvariety. Then \[\pi^*SMC_y(\Omega) = SMC_y(\pi^{-1} \Omega) \in K_T(X) [[y^{ \pm 1}]] \/.\]
\end{prop} \begin{proof}By definition, 
\[ \pi^*SMC_y(\Omega) = (-y)^{\dim \Omega} \pi^* \frac{\caD(MC_y(\Omega))}{\lambda_y(T^*_Y)} = 
(-y)^{\dim \Omega} \frac{ \pi^*(MC_y(\Omega)^\vee \otimes \omega_{Y}^\bullet)}{\pi^*(\lambda_y(T^*_Y))} \/. \] 
Here by $[\mathcal{F}]^\vee$ we denoted the automorphism of $K_T(X)[y^{\pm1}]$ which takes the class of a vector bundle 
$[E]$ to its dual $[E^\vee]$, and it sends $y \mapsto y^{-1}$. For instance, it will take $\lambda_y(T^*_\pi)$ to 
$\lambda_{y^{-1}}(T_\pi)$. Same applies to coherent sheaves, by using a (finite) resolution by vector bundles. 

Now apply Verdier-Riemann-Roch from Thm. \ref{thm:existence} to the numerator to obtain 
\[ \pi^*(MC_y(\Omega)^\vee \otimes \omega_{Y}^\bullet) = \frac{(MC_y(\pi^{-1} \Omega))^\vee \otimes \pi^*(\omega_{Y}^\bullet)}{\lambda_{y^{-1}}(T_\pi)} 
= (-1)^{\dim Y}\frac{(MC_y(\pi^{-1} \Omega))^\vee \otimes \pi^*(\wedge^{top}T^*_{Y})}{\lambda_{y^{-1}}(T_\pi)} \]
We need to analyze the class 
\[ \frac{ \pi^*(\wedge^{top}T^*_{Y})}{\pi^*(\lambda_y(T^*_Y)) \lambda_{y^{-1}}(T_\pi)} 
= \frac{ \pi^*(\wedge^{top}T^*_{Y}) \lambda_y(T^*_\pi)}{\lambda_y(T^*_X) \lambda_{y^{-1}}(T_\pi)} 
= \frac{\omega_{X}}{\lambda_y(T^*_X)} \times \frac{\lambda_y(T^*_\pi)}{\wedge^{top}(T^*_\pi) \lambda_{y^{-1}}(T_\pi)} \/. \] 
We have the following general situation. For any vector space $V$ of dimension $d$ we have a pairing 
\[ \wedge^i V \otimes \wedge^{d-i} V \simeq \wedge^d V \Leftrightarrow \wedge^i V \otimes \wedge^d V^* \simeq \wedge^{d-i} V^* \/. \] 
This implies that \[ \begin{split} [\wedge^d V^*]  \lambda_{y^{-1}}(V)&  = [\wedge^d V^*] (1+ y^{-1} [V]+ \ldots + y^{-d} [\wedge^d V]) \\ & 
= [\wedge^d V^* + y^{-1} [\wedge^{d-1} V^*] + \ldots + y^{-d}\\ & = y^{-d} \lambda_y (V^*) \/. \end{split}\] 
Therefore \[ \frac{\lambda_y(T^*_\pi)}{\wedge^{top}(T^*_\pi) \lambda_{y^{-1}}(T_\pi)} = y^{\dim X - \dim Y} \/. \] 
But $\omega_{X}^\bullet = (-1)^{\dim X} \omega_{X}$, and $\dim \pi^{-1} (\Omega) = \dim \Omega + \dim X - \dim Y$ 
(because $\pi$ is smooth) and we combine all of the above to obtain 
\[(-y)^{\dim \Omega} \frac{ \pi^*(MC_y(\Omega)^\vee \otimes \omega_{Y}^\bullet)}{\pi^*(\lambda_y(T^*_Y))} 
= (-y)^{\dim \Omega+ \dim X - \dim Y} \frac{ MC_y(\pi^{-1}(\Omega))^\vee \otimes \omega_X^\bullet}{\lambda_y(T^*_X)}\/, \] 
which is the definition of the Segre class $SMC_y(\pi^{-1} \Omega)$.\end{proof}

\section{Motivic Chern classes of Schubert cells and left DL operators} 
In this section we prove that the left Demazure-Lusztig operators generate recursively the motivic Chern classes for Schubert cells. 
For the right operators, this was proved in \cite{AMSS:motivic}.

The equivariant motivic Chern classes of the Schubert cells $\{MC_y(X_w^\circ)\mid w\in W\}$ form a basis for the localized K-group $K_T(G/B)_{\loc}$.~These classes are equivalent (in a certain sense) to the K-theoretic version of the stable envelopes \cite{MR3752463,su2020k}, and may be used to determine certain Whittaker vectors, and the `standard basis' in the unramified principal series representation of $p$-adic groups \cite{MSA:whittaker, AMSS:motivic}.

\subsection{Poincar{\'e} duality} In this section we investigate the Poincar{\'e} duals of the motivic Chern classes in $G/P$ with respect to the K-theoretic pairing. The starting point is the following result for $G/B$, proved in \cite[Theorem 8.11]{AMSS:motivic}. The proof is based on the relation between the motivic Chern classes and stable envelopes \cite{feher2021motivic,AMSS:motivic}.
\begin{thm}\label{thm:dual2} The following holds in $K_T(G/B)[y,y^{-1}]$: \[ \langle MC_y(X_w^\circ), SMC_y(X^{u, \circ}) \rangle = \delta_{u,w} \/. \]
\end{thm}
It turns out that this result extends without change to $G/P$.
\begin{thm}\label{thm:dualP}
Let $u,w \in W^P$. Then the motivic Chern classes are dual to the Segre motivic Chern, i.e. \[ \langle MC_y(X_{wW_P}^\circ), SMC_y(X^{uW_P,\circ}) \rangle_{G/P} = \delta_{u,w} \/. \]
\end{thm}

\begin{proof} Let $\pi: G/B \rightarrow G/P$ be the natural projection. This is a smooth morphism, to which we apply the Proposition \ref{prop:gensegrepb}. Observe that since $w\in W^P$, the restriction $\pi: X_w \to X_{wW_P}$ is birational over the Schubert cell $X_{w W_P}^\circ$. Thus by functoriality $\pi_* (MC_y(X_w^\circ))=MC_y(X_{wW_P}^\circ)$. By projection formula we have \[ \begin{split} &\langle MC_y(X_{wW_P}^\circ), SMC_y(X^{uW_P,\circ}) \rangle_{G/P}\\
 = & \int_{G/P} \pi_* (MC_y(X_{w}^\circ))\cdot SMC_y(X^{uW_P,\circ}) \\  =& 
 \int_{G/B} MC_y(X_w^\circ) \cdot  \pi^*(SMC_y(X^{uW_P,\circ})) \\  = 
 &\sum_{v\in uW_P} \int_{G/B} MC_y(X_w^\circ) \cdot  SMC_y(X^{v,\circ}) \\  = 
 & \delta_{w,u} \/.\end{split}\]  
Here the third equality follows from Proposition \ref{prop:gensegrepb} using that $\pi^{-1}( X^{u W_P} )= \coprod_{v \ge u} X^{v, \circ}$, and the fourth is a consequence of Poincar{\'e} duality for $G/B$, from Theorem \ref{thm:dual2}. 
\end{proof}

\subsection{Motivic Chern classes of Schubert cells via left DL operators} In this section we prove formulae for the action of the DL operators on motivic Chern (MC) classes and Segre motivic classes (SMC). We will start by recalling the action of the right DL operators.
\begin{thm}\cite{AMSS:motivic}\label{thm:MCR}
For $w \in W$ and simple root $\alpha_i$, we have  
\[ \mathcal{T}^R_i MC_y(X_w^\circ) = \begin{cases} MC_y(X_{ws_i}^\circ) &\textit{ if } ws_i>w,\\
-(1+y)MC_y(X_w^\circ)-yMC_y(X_{ws_i}^\circ) &\textit{ if } ws_i<w, \end{cases}\]
and
\[ \mathcal{T}^R_i MC_y(X^{w,\circ}) = \begin{cases} MC_y(X^{ws_i,\circ}) &\textit{ if } ws_i<w,\\
-(1+y)MC_y(X^{w,\circ})-yMC_y(X^{ws_i,\circ}) &\textit{ if } ws_i>w. \end{cases}\]
In particular, $MC_y(X_w^\circ)=\mathcal{T}^R_{w^{-1}}(\calO_{id})$, and $MC_y(X^{w,\circ})=\mathcal{T}^R_{w^{-1}w_0}(\calO^{w_0})$.
\end{thm}
\begin{proof}
The first equality follows from \cite{AMSS:motivic}, while the second one follows from Lemma \ref{lemma:Kcomm}(a) and $w_0^LMC_y(X_w^\circ)=MC_y(X^{w_0w,\circ})$.
\end{proof}

For the Segre motivic Chern classes, we have the following theorem.
\begin{thm}\label{thm:SMCR}
For $w \in W$ and simple root $\alpha_i$, we have  
{\[ \mathcal{T}^{R,\vee}_i SMC_y(X_w^\circ) = \begin{cases} -ySMC_y(X_{ws_i}^\circ) &\textit{ if } ws_i<w,\\
-(1+y) SMC_y(X_w^\circ) +SMC_y(X_{ws_i}^\circ) &\textit{ if } ws_i>w \/; \end{cases}\]}
and
{\[ \mathcal{T}^{R,\vee}_i SMC_y(X^{w,\circ}) = \begin{cases} -ySMC_y(X^{ws_i,\circ}) &\textit{ if } ws_i>w,\\
-(1+y)SMC_y(X^{w,\circ})+SMC_y(X^{ws_i,\circ}) &\textit{ if } ws_i<w. \end{cases}\]}
In particular, 
 \[SMC_y(X_w^\circ)= \frac{(-y)^{\ell(w)}}{\prod_{\alpha>0}(1+ye^\alpha)}(\calT_w^{R,\vee})^{-1}(\calO_{id}) \/; \] \[ SMC_y(X^{w,\circ})=\frac{(-y)^{\dim G/B-\ell(w)}}{\prod_{\alpha>0}(1+ye^{-\alpha})}(\calT_{w_0w}^{R,\vee})^{-1}(\calO^{w_0})\/.\]
\end{thm}
\begin{proof}
It suffices to prove the second equality, as the first one can be deduced by applying $w_0^L$. To prove the second equality, 
we utilize Theorem \ref{thm:dual2} and Lemma \ref{lemma:Kcomm}(c) to obtain:
\[ \begin{split}
\mathcal{T}^{R,\vee}_i SMC_y(X^{w,\circ}) =& \sum_{u\in W}\langle \mathcal{T}^{R,\vee}_i SMC_y(X^{w,\circ}), MC_y(X_u^\circ)\rangle SMC_y(X^{u,\circ})\\
=& \sum_{u\in W}\langle  SMC_y(X^{w,\circ}), \mathcal{T}^R_i MC_y(X_u^\circ)\rangle SMC_y(X^{u,\circ}) \/. \end{split}\] 
Then from Theorem \ref{thm:MCR}, the last expression equals
\[ \begin{split} & \sum_{u\in W, us_i>u}\langle  SMC_y(X^{w,\circ}),  MC_y(X_{us_i}^\circ)\rangle SMC_y(X^{u,\circ})
\\ & +\sum_{u\in W, us_i<u}\langle  SMC_y(X^{w,\circ}),  -(1+y)MC_y(X_u^\circ)-yMC_y(X_{us_i}^\circ)\rangle SMC_y(X^{u,\circ})\\
& =\begin{cases} -ySMC_y(X^{ws_i,\circ}) &\textit{ if } ws_i>w,\\
-(1+y)SMC_y(X^{w,\circ})+SMC_y(X^{ws_i,\circ}) &\textit{ if } ws_i<w \/.\end{cases}
\end{split}\]
Finally, the last two identities follow from \[ SMC(X_{id}^\circ)=\frac{\calO_{id}}{\prod_{\alpha>0}(1+ye^\alpha)} \textrm{ and } SMC(X^{w_0,\circ})=\frac{\calO^{w_0}}{\prod_{\alpha>0}(1+ye^{-\alpha})} \/. \]
\end{proof}

We continue towards investigating the actions of the left DL operators on motivic Chern and Segre classes for any $G/P$. 
The ideas of proofs are the same as in cohomology, but the formulae get more involved. We need the following lemma.
\begin{lem}\cite[Lemma 2.1]{deodhar:bruhatII}\label{lem:minimal}
For any simple root $\alpha_i$ and $w\in W^P$, exactly one of the following occurs:
\begin{enumerate}
\item $s_iw<w$, in which case, $s_iw\in W^P$ as well;
\item $s_iw>w$ and $s_iw\in W^P$;
\item $s_iw>w$ and $s_iw\notin W^P$. In this case, $s_iw=ws_j$ for some simple reflection  $s_j$ in $W_P$.
\end{enumerate}
\end{lem}

First of all, for the motivic Chern classes of the Schubert cells, we have the following theorem.
\begin{thm}\label{thm:LDLMCP}
For any $w\in W^P$ and any simple reflection $s_i$, we have
\[\mathcal{T}_i^L(MC_y(X_{wW_P}^\circ))=\begin{cases} (-y)^{\ell(s_iw)-\ell(s_iwW_P)}MC_y(X_{s_i wW_P}^\circ) &\textit{ if } s_iw>w;\\
-(1+y)MC_y(X_{wW_P}^\circ)-y MC_y(X_{s_iwW_P}^\circ) &\textit{ if } s_iw<w. \end{cases}\]
\end{thm}
\begin{remark}
Using $w_0^L\mathcal{T}_{\alpha_i}^Lw_0^L=\calT_{-w_0\alpha_i}^{L,\vee}$ and $w_0^LMC_y(X_{wW_P}^\circ)=MC_y(X^{w_0wW_P,\circ})$, we can get a similar formula for the action of $\mathcal{T}_{\alpha_i}^{L,\vee}$ on $MC_y(X^{wW_P,\circ})$. Same comment applies to the action of $\mathcal{T}_{i}^L$ on $SMC_y(X_{wW_P}^\circ)$ in Theorem \ref{thm:SMCTL} below.
\end{remark}
\begin{proof} We first consider the case when $P=B$. In this case, due to the quadratic relation $(\mathcal{T}_i^L+1)(\mathcal{T}_i^L+y)=0$, the statement for the case $s_iw<w$ can be deduced from the case $s_iw>w$. In the latter case,
by Theorem \ref{thm:MCR},
\[MC_y(X_{s_i}^\circ)=\mathcal{T}^R_{i}(\cO_{id})=(1+y\calL_{\alpha_i})\cO_{s_i}-\cO_{id}=(1+ye^{-\alpha_i})\cO_{s_i}-(1+y+ye^{-\alpha_i})\cO_{id},\]
where the last equality can be easily checked by restricting to the fixed points $e_{id}$ and $e_{s_i}$.
On the other hand, by definition of $\mathcal{T}_i^L$ and Proposition \ref{prop:Kaction},
\begin{align*}
\mathcal{T}_i^L(\calO_{id})&=(1+ye^{-\alpha_i})\calO_{s_i}-(1+y+ye^{-\alpha_i})\calO_{id}
=MC_y(X_{s_i}^\circ).
\end{align*}
Therefore,
\begin{align*}
T^L_i MC_y(X_w^\circ)& =T^L_i \mathcal{T}^R_{w^{-1}}(\cO_{id})= \mathcal{T}^R_{w^{-1}}T^L_i(\cO_{id})  \\ & = \mathcal{T}^R_{w^{-1}}\mathcal{T}^R_i(\cO_{id})=\mathcal{T}^R_{w^{-1}s_i}(\cO_{id})=MC_y(X_{s_i w}^\circ) \/, 
\end{align*}
where the second equality follows from Lemma \ref{lemma:Kcomm}(a). This finishes the case $P=B$.

Finally, we turn to the parabolic case. Recall that $\pi:G/B\rightarrow G/P$ is the natural projection, and that for any $u\in W$, we have (see \cite[Remark 5.5]{AMSS:motivic})
\[\pi_*MC_y(X_u^\circ)=(-y)^{\ell(u)-\ell(uW_P)}MC_y(X_{uW_P}^\circ).\]
Thus, for $w \in W^P$, and by Proposition \ref{prop:Kprop}(d) and Lemma \ref{lem:minimal}(1),
\begin{align*}
\mathcal{T}_i^L(MC_y(X_{wW_P}^\circ))=&\mathcal{T}_i^L\pi_*MC_y(X_{w}^\circ)=\pi_*\mathcal{T}_i^LMC_y(X_{w}^\circ)\\
=&\begin{cases} (-y)^{\ell(s_iw)-\ell(s_iwW_P)}MC_y(X_{s_i wW_P}^\circ) &\textit{ if } s_iw>w,\\
-(1+y)MC_y(X_{wW_P}^\circ)-y MC_y(X_{s_iwW_P}^\circ) &\textit{ if } s_iw<w. \end{cases}
\end{align*} This finishes the proof.\end{proof}

For the Segre motivic Chern classes, we have the following theorem.
\begin{thm}\label{thm:SMCTL}
For any $w\in W^P$ and a simple reflection $s_i$, the following holds
\begin{align*}
\mathcal{T}_i^{L,\vee}(SMC_y(X^{wW_P,\circ}))=\begin{cases} -ySMC_y(X^{s_iwW_P,\circ}) & \textit{ if } s_iw>w;\\
-(1+y)SMC_y(X^{wW_P,\circ})+SMC_y(X^{s_iwW_P,\circ}) & \textit{ if } s_iw<w.
\end{cases}
\end{align*}
\end{thm}
\begin{proof} Due to the quadratic relation $(\mathcal{T}_i^{L,\vee}+1)(\mathcal{T}_i^{L,\vee}+y)=0$ and since $s_i w <w$ implies that $s_iw \in W^P$ by Lemma \ref{lem:minimal}, the statement for $s_iw<w$ can be deduced from the case $s_iw>w$. We assume $s_iw>w$. For any $u\in W^P$, we have
\begin{align*}
&\langle\mathcal{T}_i^{L,\vee}(SMC_y(X^{wW_P,\circ})),MC_y(X_{uW_P}^\circ)\rangle\\
=&s_i\langle SMC_y(X^{wW_P,\circ}),\mathcal{T}_i^{L}MC_y(X_{uW_P}^\circ)\rangle\\
=&\begin{cases} s_i\langle SMC_y(X^{wW_P,\circ}),(-y)^{\ell(s_iu)-\ell(s_iuW_P)}MC_y(X_{s_i uW_P}^\circ)\rangle &\textit{ if } s_iu>u;\\
s_i\langle SMC_y(X^{wW_P,\circ}),-(1+y)MC_y(X_{uW_P}^\circ)-y MC_y(X_{s_iuW_P}^\circ)\rangle &\textit{ if } s_iu<u, \end{cases}\\
=&\begin{cases} 0 &\textit{ if } s_iu>u, \textit{ and } s_iu\in W^P;\\
-y\delta_{w,u} &\textit{ if } s_iu>u, \textit{ and } s_iu\notin W^P;\\
-y\delta_{w,s_iu} &\textit{ if } s_iu<u, \end{cases}\\
=&-y\delta_{s_iwW_P,uW_P}.
\end{align*}
where the first equality follows from Lemma \ref{lemma:Kcomm}(c), the second one follows from Theorem \ref{thm:LDLMCP}, 
and the third and the last one follow from Theorem \ref{thm:dualP} and Lemma \ref{lem:minimal}. The theorem follows 
from this and Theorem \ref{thm:dualP}.\end{proof}

\section{Left operators in quantum cohomology and quantum K-theory}\label{s:quantumleft}
In this section we extend the definition of left divided difference operators to equivariant quantum cohomology
and equivariant quantum K-theory. 

\subsection{Equivariant quantum cohomology} Our main reference for quantum cohomology is \cite{fulton.pandh:notes}.
~As above $X$ is a smooth, complex, projective variety with an action of a torus $T$. 
Assume also that $X$ is Fano and that there is a finite basis $\gamma_1, \ldots , \gamma_N \in H_2(X)$ 
for the effective cone of curves in $X$. A {\em degree} $d=(d_1, \ldots , d_N)$ is an element 
$d = d_1 \gamma_1 + \ldots d_N \gamma_N \in H_2(X)$. 
The $T$-equivariant quantum 
cohomology ring $\QH^*_T(X)$ is a graded algebra over $H^*_T(X)$. Additively, 
$\QH^*_T(X) = H^*_T(X) \otimes_{H^*_T(pt)} H^*_T(pt)[q_1, \ldots , q_N]$ 
where $q:=(q_1, \ldots, q_N)$ is the sequence of (quantum) parameters indexed by the 
basis of the effective cone of curves in $X$. The degree of $q_i$ is \[ \deg (q_i) := \int_X c_1(T_X) \cap \gamma_i \quad \/. \]

The quantum multiplication, denoted by $\star$, is determined by the {\em equivariant Gromov-Witten} (GW) 
invariants defined by Givental \cite{givental:egw}. More precisely, recall that $\langle \cdot , \cdot \rangle$ 
is the intersection pairing, which we now extend
by $\Q[q]$-linearity. Then for any $a, b \in H^*_T(X)$, the quantum multiplication is 
determined by the condition that for all $c \in H^*_T(X)$, \begin{equation}\label{E:qpairing} \langle a \star b, c \rangle = \sum_{d \ge 0} \langle a , b , c \rangle_d q^d \/.\end{equation} The equivariant GW invariants $\langle a , b , c \rangle_d$ are defined by \[ \langle a , b , c \rangle_d = \int_{[\overline{M}_{0,3}(X,d)]^{vir}} \ev_1^*(a) \cdot \ev_2^*(b) \cdot \ev_3^*(c) \/, \] where $[\overline{M}_{0,3}(X,d)]^{vir}$ is the virtual fundamental class associated to Kontsevich moduli space of stable maps $\overline{M}_{0,3}(X,d)$, and $\ev_i:\overline{M}_{0,3}(X,d) \to X$ are the corresponding evaluation maps. If $X=G/P$ is a flag manifold, then the moduli space is irreducible
with quotient singularities \cite{thomsen:irreducibility}, and the virtual class $[\overline{M}_{0,3}(X,d)]^{vir}$ is the usual fundamental class.

The (equivariant) quantum cohomology ring is not functorial with respect to morphisms $f: X \to Y$, however, it {\em is} functorial if $f$ is an isomorphism. More precisely, an isomorphism $f: X \to Y$ induces an isomorphism of the moduli spaces $\overline{f}:\overline{\mathcal{M}}_{0,3}(X,d) \to \overline{\mathcal{M}}_{0,3}(Y,f_*(d))$. If $X,Y$ are $T$-varieties then $f^*$ also induces isomorphisms of the equivariant cohomology rings of $X,Y$, and of the moduli spaces. From this we deduce that there is a $\Q[q]$-linear isomorphism $f^*: \QH^*_T(Y) \to \QH^*_T(X)$ between the (equivariant) quantum cohomology rings, {where the $T$-action on $Y$ is induced via $f$.}



\subsection{Quantum left divided difference operators} From now on we specialize to the situation 
$X^P=G/P$ is a flag manifold. A basis for the effective cone of curves is given by the Schubert classes $[X_{s_iW_P}]$, where $s_i \in W^P$. The ring $\QH^*_T(X^P)$ is a free $H^*_T(pt)[q]$-algebra with a basis given by Schubert classes $[X^{wW_P}] \otimes 1$; for simplicity we still denote these classes by $[X^{wW_P}]$.

Fix a Weyl group element $w \in W$. The left Weyl group multiplication $w^L$ is an automorphism of $X^P$, and, as explained above, it induces an automorphism of the moduli space $\Mb_{0,3}(X,d)$ by acting on the target of a stable map. More precisely, for $w \in W$ and 
$[C, p_1, p_2,p_3; \varphi] \in \Mb_{0,3}(X,d)$ the class of a stable map $\varphi:(C,p_1,p_2,p_3) \to X^P$,  
\[ w^L. [C, p_1, p_2,p_3; \varphi] := [C,p_1,p_2,p_3; w^L. \varphi] \/, \] where $(w^L.\varphi)(x) = w^L . (\varphi(x))$. 
Since the evaluation maps are $G$-equivariant, it follows that the corresponding automorphism of the equivariant cohomology ring 
$H^*_T(\Mb_{0,3}(X,d))$ satisfies \[ w^L . \ev_i^*(a) = \ev_i^*(w^L.a) \/,~ \forall a \in H^*_T(\Mb_{0,3}(X,d)) \/. \] From this we deduce that $w^L$ acts on equivariant GW invariants by 
\begin{equation*}\label{E:wGW} w^L. \langle a_1,a_2,a_3 \rangle_d = \langle w^L(a_1),w^L(a_2), w^L(a_3) \rangle_d \quad ~, \forall a_1,a_2,a_3 \in H^*_T(X^P) \/.\end{equation*} 
Then the left action of $W$ on $H^*_T(X^P)$ extends by $\Q[q]$-linearity to $\QH^*_T(X^P)$, and each 
$w \in W$ induces a $\Q[q]$-linear {\em ring} automorphism of the quantum cohomology $\QH^*_T(X^P)$: 
\begin{equation}\label{E:qhw} w^L. (a_1 \star a_2) = (w^L.a_1) \star (w^L .a_2) \/. \end{equation} 
It makes sense to define the (quantum) left divided difference operator as in equation \eqref{E:leftcohdiv} above:
\[ \delta_i := \frac{1}{\alpha_i}(id - s_i^L) \/. \] We summarize below its properties.

\begin{prop}\label{prop:quantumdelta} The operators $\delta_i: \QH^*_T(X^P) \to \QH^*_T(X^P)$ satisfy the following properties:

(a) $\delta_i$ is linear over $\Q[q]$ and the quantum operators $\delta_i$ satisfy the same quadratic and braid relations as the ordinary operators; see \S \ref{sec:cohdivdiff}.

(b) For each $w \in W^P$, \[ \delta_i [X^{wW_P}]= \begin{cases} [X^{s_i w W_P}] & \textrm{if } s_i w < w \/; \\ 0 & \textrm{otherwise} \end{cases} \]

(c) $\delta_i$ satisfies the (quantum) Leibniz rule i.e. for any $a,b \in \QH^*_T(X^P)$, \[ \delta_i(a \star b) = \delta_i(a) \star b + s_i^L(a) \star \delta_i(b) \/. \]

(d) The operator $\delta_i$ is a $\QH_G(G/P)$-module homomorphism, i.e. for any $\kappa \in \QH_G(G/P)$ and $\eta \in \QH_T(G/P)$, \[ \delta_i(\kappa \star \eta) = \kappa \star \delta_i(\eta) \/. \]

\end{prop}
\begin{proof} By definition, the quantum operator $\delta_i$ is just the $\Q[q]$-linear extension of the ordinary operator on $H^*_T(X^P)$. Then (a) and (b) follow from the corresponding statements in the non-quantum case. Part (c) follows because $s_i^L$ is a 
ring automorphism of $\QH^*_T(X^P)$ and part (d) because $s_i^L$ is $\QH^*_G(X^P)$-linear.
\end{proof}



Let $w^P$ be the longest element in $W^P$. An immediate consequence of the proposition is that any quantum Schubert class may be obtained from the class of the point $[e_{w^P W_P}]$ by successively applying the left divided difference operators. If one knows a presentation of the equivariant quantum ring by generators and relations, this may be used to construct `double quantum Schubert polynomials' which represent Schubert classes. The initial step requires identifying the (quantum) class of the point, which is a nontrivial problem. In this context, the left divided difference operators were utilized by I. Ciocan-Fontanine \cite[App. J]{fulton.pragacz}, and Kirillov and Maeno \cite{kirillov.maeno:quantum-schubert} 
to investigate double quantum Schubert polynomials for a presentation of $\QH^*_T(\mathrm{SL}_n(\C)/B)$. {Lenart and Maeno 
\cite{lenart.maeno:quantum} generalized this to double quantum Grothedieck polynomials, and equivariant quantum K-theory.}
Observe that Proposition \ref{prop:quantumdelta} does not require the knowledge of a presentation.

\begin{example}\label{ex:qhgr24} Let $X^P := \mathrm{SL}_4(\C)/P = \Gr(2,4)$, the Grassmann manifold parametrizing linear subspaces of dimension $2$ in $\C^4$. In this case $P$ is the maximal parabolic group so that $W_P = \langle s_1,s_3 \rangle$. Let $\Delta = \{ \alpha_1, \alpha_2, \alpha_3 \}$ be the positive simple roots. The ring $\QH_T^*(\Gr(2,4))$ is generated as a $H^*_T(pt)[q]$-algebra by $\sigma_1 := [X^{s_2 W_P}]$ and $\sigma_{1,1}:=[X^{s_1 s_2 W_P}]$. (These 
correspond to the $B^-$-Schubert classes indexed by partitions $(1)$, respectively $(1,1)$.) By direct calculation using software such as the {\em Equivariant Schubert Calculator}\begin{footnote}{The program is available at \texttt{http://sites.math.rutgers.edu/\~{}asbuch/equivcalc/} }\end{footnote} the class of the point is \[ [X^{s_2 s_1 s_3 s_2 W_P}] = \sigma_{1,1} \star \sigma_{1,1} - \alpha_1 \sigma_1 \star \sigma_{1,1} \/. \] One obtains that \[ \delta_2(\sigma_1) = id; \quad  \delta_1(\sigma_{1,1}) = \sigma_1 \/; \quad \delta_1(\sigma_1) = \delta_3(\sigma_1) = \delta_2(\sigma_{1,1}) = \delta_3(\sigma_{1,1}) = 0 \/. \] From this and the Leibniz formula one obtains that \[ \begin{split} [X^{s_1 s_3 s_2 W_P}]  & = \delta_2 (\sigma_{1,1} \star \sigma_{1,1} - \alpha_1 \sigma_1 \star \sigma_{1,1}) \\ & = -\langle \alpha_1, \alpha_2^\vee \rangle \sigma_1 \star \sigma_{1,1} - s_2(\alpha_1) \delta_2(\sigma_1 \star \sigma_{1,1}) \\ & = \sigma_1 \star \sigma_{1,1} - (\alpha_1 +\alpha_2) \sigma_{1,1} \/. \end{split} \]
More generally, a determinantal formula for the equivariant quantum class of the point (in fact, for any Schubert class) in any Grassmannian is calculated in \cite{mihalcea:giambelli}. Interestingly, the class is independent of the quantum parameter $q$, and by Proposition \ref{prop:quantumdelta}, this implies that one may find polynomials representing equivariant quantum Schubert classes which are independent of $q$.~Similar results hold for the maximal orthogonal Grassmannians; see \cite{IMN:factorial}. 
\end{example}

\subsection{Quantum K-theory} Same ideas extend to the equivariant quantum K-theory ring. We sketch the construction below, 
in the case of the flag manifolds $X^P=G/P$. Our main references for (equivariant) quantum K-theory ring are \cite{lee:QK,buch.m:qk}.
We continue the notation from the previous section. 

Additively, the quantum K-theory ring is 
\[ \QK_T(X^P) = K_T(X^P) \otimes_{K_T(pt)} K_T(pt) [[q]] \/.\] In particular, it is a free $K_T(pt) [[q]]$-module with basis 
$\{ \cO^w \otimes 1 \}_{w \in W^P}$; for simplicity we keep the `classical' notation $\cO^w$ for the basis elements. There is a ring structure defined by 
Givental and Lee \cite{givental:onwdvv,lee:QK}; we denote by `$\circ$' the (equivariant) quantum K multiplication. The structure constants are determined by the K-theoretic Gromov-Witten invariants, defined by taking sheaf 
Euler characteristics on the moduli space of stable maps $\Mb_{0,3}(X^P,d)$. A single such structure constant is of the form 
\[ \chi(\Mb_{0,3}(X^P,d); \ev_1^*(a) \cdot \ev_2^*(b) \cdot \ev_3^*(c) \cdot \cO_{\Mb_{0,3}(X^P,d)}(-\partial \Mb)) \/, \] where 
$\partial \Mb$ is a certain divisor included in the boundary of the moduli space, independent of the classes $a,b,c$; see \cite{buch.m:qk}. It was proved in \cite{buch.m:qk,BCMP:qkfin} that quantum K multiplication is finite if $X^P$ is a cominuscule Grassmannian; recently this was 
generalized to any flag manifold by \cite{ACTI:finiteness,kato:loop}. Therefore for flag manifolds $X^P$ we may replace the power series ring 
$K_T(pt)[[q]]$ by the polynomial ring $K_T(pt)[q]$.

As in the quantum cohomology case, the 
equivariant quantum K-theory is functorial for isomorphisms. In fact, the same argument as in cohomology shows that the left action of 
$W$ on $K^*_T(X^P)$ extends by $\Q[q]$-linearity to $\QK_T(X^P)$, and that each 
$w \in W$ induces a $\Q[q]$-linear ring automorphism of $\QK_T(X^P)$. 
Then one defines the (quantum) left divided difference operators on $\QK_T(X^P)$ as in equation \eqref{E:KleftDem} above:
\begin{equation}\label{E:qDem}  \delta_i := \frac{1}{1- e^{\alpha_i}}(id - e^{\alpha_i} s_i^L) \/; \quad \delta_i^\vee := \frac{1}{1- e^{-\alpha_i}}(id - e^{-\alpha_i} s_i^L) \/. \end{equation} 
These operators have the same properties as the ordinary ones, except that the ordinary K-product is now replaced by the 
quantum K product. For reader's convenience, we state these properties; the proofs are the same as for the 
classical operators, using that $s_i^L$ is a $\Q[q]$-linear ring automorphism of $\QK_T(X^P)$. We use the dual operators, as they behave 
well with respect to opposite Schubert classes (cf.~ Remark \ref{rmk:dualDem}).
\begin{prop}\label{prop:QKDem} (a) The quantum operators $\delta_i^\vee$ are $\Q[q]$-linear, and satisfy the braid relations, and $(\delta_i^\vee)^2= \delta_i^\vee$. 

(b) For each $w \in W^P$, \[ \delta_i^\vee (\cO^{wW_P})= \begin{cases} \cO^{s_i w W_P} & \textrm{if } s_i w < w \/; \\ \cO^{wW_P} & \textrm{otherwise} \/. \end{cases} \]

(c) (Leibniz rule) For any $a, b \in \QK_T(X)$, \[ \delta_i^\vee(a \circ b) = \delta_i^\vee(a) \circ b + e^{-\alpha_i} s_i^L(a) \circ \delta_i^\vee(b) - e^{-\alpha_i} s_i^L(a) \circ s_i^L(b) \/. \]

(d) The operator $\delta_i^\vee$ is a $\QK_G(G/P)$-module homomorphism, i.e. for any $\kappa \in \QK_G(G/P)$ and $\eta \in \QK_T(G/P)$, \[ \delta_i^\vee(\kappa \circ \eta) = \kappa \circ \delta_i(\eta) \/. \]
\end{prop}
\begin{example}\label{ex:QKGr} We upgrade Example \ref{ex:qhgr24} to the ring $\QK_T(\Gr(2,4))$. Let $\cO_1:= \cO^{s_2 W_P}$ and $\cO_{1,1}:= \cO^{s_1 s_2 W_P}$. These classes generate the ring over $K_T(pt)[q]$. For a positive simple root $\alpha_i$, denote by $T_i := e^{\alpha_i}$. Using again the {\em Equivariant Schubert Calculator} (this time based on results from \cite{buch.m:qk}), one calculates that the class of the point in $\QK_T(\Gr(2,4))$ is: \[ \cO^{s_2s_1 s_3 s_2W_P} = T_1^{-1} \cO_{1,1} \circ \cO_{1,1} + (1- T_1^{-1}) \cO_1 \circ \cO_{1,1} \/. \] One calculates that \[ \delta_2^\vee (\cO_{1,1}) = s_2^L (\cO_{1,1}) = \cO_{1,1}; \quad \delta_2^\vee(\cO_1) = 1\/; \quad s_2^L(\cO_1) = T_2^{-1} \cO_1 + 1 - T_2^{-1} \/. \] Using the Leibniz rule one calculates that $\delta_2^\vee(T_1^{-1} \cO_{1,1} \circ \cO_{1,1}) = 0$ and that \[ \delta_2^\vee((1- T_1^{-1}) \cO_1 \circ \cO_{1,1}) = (T_1 T_2)^{-1} \cO_1 \circ \cO_{1,1} + (1- (T_1T_2)^{-1}) \cO_{1,1} \/. \] One may check directly that the latter class equals to $\cO^{s_1s_3 s_2 W_P}$, as claimed in Proposition \ref{prop:QKDem}.
\end{example}

\begin{remark} The definition of the quantum Demazure operators from \eqref{E:qDem} makes sense for the equivariant K-theory $\QK_T(X)$ of any complex, projective variety $X$ with a $G$-action. In the non-quantum case, left Demazure operators on $K_T(X)$ are constructed in the Appendix below, and coincide with those from \cite{harada.landweber.sjamaar:divided}. The quantum generalization follows the same arguments utilized in the $G/P$ case.

Further, using the definitions from \S \ref{ss:leftDL}, one may also define a quantum version of the DL operators on $\QK_T(X)[y]$. If $X=G/P$, this will satisfy 
an analogue of Theorems \ref{thm:LDLMCP} and \ref{thm:SMCTL}. At this time we do not know a good geometric interpretation for the structure constants in the quantum multiplication of motivic Chern classes, or Segre motivic classes, of Schubert cells.\end{remark} 

\section{Appendix: Demazure operators via convolution}\label{appendix} In this appendix we generalize the construction of the left Demazure operators to the $T$-equivariant $K$-theory of any quasi projective variety $X$ which admits a $G$-action. Then we show that these operators are equal to certain convolution operators. If $X=G/B$, we also construct both the left and right Demazure operators via convolutions. This will be useful in relating the operators defined in this paper (via left/right $W$-action) to the convolution operators such as those in \cite[Remark 3.3]{lenart.zhong.etal:parabolic}. As shown in Proposition \ref{prop:eqdelta}, our construction gives the same operators as the ones defined in \cite{harada.landweber.sjamaar:divided}. {All these constructions specialize to the equivariant cohomology setting.}


\subsection{Left Demazure operator in the general setting}
In this appendix, a {\em variety} is a complex quasi-projective integral scheme; all algebraic groups are defined over $\C$. We recall few facts about the `induction' and 'restriction' functors; we refer to \cite[\S2]{brion:multiplicity} for details; see also \cite[\S 5.2.16]{chriss2009representation}. For a variety $X$ with an action of an algebraic group $H$, we denote by $Sh_H(X)$ the abelian category of $H$-linearized sheaves on $X$. 

Consider the algebraic equivariant $K$-theory group $K_H(X)$ consisting of the Grothendieck group of $H$-linearized coherent sheaves on $X$. The tensor product gives $K_H(X)$ a module structure over $K^H(X)$, the Grothendieck ring of $H$-equivariant vector bundles on $X$. We refer to \cite{chriss2009representation,harada.landweber.sjamaar:divided} for details. 

Fix $ T \subset B \subset G$ as in section \S \ref{sec:flags}, where we assume in addition that $G$ is simply connected. Consider a variety $X$ with a $B$-action. The restriction map induces an isomorphism $K_B(X) = K_T(X)$; see \cite[\S5.2.18]{chriss2009representation}. 
Let $B$ act on $G \times X$ by $b.(g,x) = (gb^{-1}, bx)$, and let $G \times^B X $ denote the quotient of $G\times X$ by this free action of $B$. We denote by $[g,x]$ the equivalence class of $(g,x)$. The variety $G \times^B X$ has a $G$ action given by left multiplication and it comes equipped with a projection morphism $pr: G \times^B X \to G/B$ sending $[g,x] \mapsto gB$. The fibre $pr^{-1}(1.B) \simeq X$ and we denote by $\iota: X \hookrightarrow G \times^B X$ the inclusion. We need the following fact - cf.~\cite[Lemma 2]{brion:multiplicity}.

\begin{lemma}\label{lemma:catequiv} Let $\mathcal{F}$ be a $B$-linearized sheaf on $X$ and let $\mathcal{G}$ be a $G$-linearized sheaf of $G \times^B X$. Then there are isomorphisms \[ \mathcal{F} \simeq \iota^* (G \times^B \mathcal{F}) \/; \quad \mathcal{G} \simeq G \times^B \iota^* \mathcal{G} \/, \] where the first isomorphism is as $B$-linearized sheaves on $X$ and the second as $G$-linearized sheaves on $G \times^B X$.
In particular, this induces an equivalence of categories $\Phi: Sh_B(X) \simeq Sh_G (G \times^B X)$ given by `induction' functor $\Phi (\mathcal{F}) = G \times^B \mathcal{F}$ and with inverse given by `restriction' $\mathcal{G} \mapsto \iota^* \mathcal{G}$. 
\end{lemma}

Assume now in addition that $X$ is a $G$-variety. Then we have in addition a multiplication morphism $m: G \times^B X \to X$ given by $[g,x] \mapsto g.x$. \[ \xymatrix{ X \ar[rr]^{\iota} \ar[d] & &  G \times^B X \ar[d]^{pr} \ar[rr]^{m} & & X \\ 1.B \ar[rr] && G/B &&}\] In this case $(pr,m): G \times^B X \to G/B \times X$ is an isomorphism of varieties and $m \circ \iota = id_X$. Using these facts and Lemma \ref{lemma:catequiv} one may prove the following Lemma. We leave the proof details to the reader. 
\begin{lemma}\label{lemma:product} (a) Let $(\mathbb{C}_\lambda)_X$ be the trivial $B$-module on $X$ with weight $\lambda$ and $\mathcal{L}_\lambda = G \times^B \mathbb{C}_\lambda$ (a line bundle on $G/B$). Then \[ G \times^B (\mathbb{C}_\lambda)_ X = pr^* \mathcal{L}_\lambda \/. \]

(b) For any $G$-linearized sheaf $\mathcal{F}$ on $X$, we have \[ m^*\mathcal{F} \simeq G \times^B \mathcal{F} \/. \]

(c) Let $\mathcal{F}$ be a $G$-linearized sheaf on $X$. Then the isomorphism $\Psi: K_B(X) \to K_G(G/B \times X)$ induced by $(pr \times m) \circ \Phi$ sends \[ \Psi [(\mathbb{C}_\lambda)_X \otimes \mathcal{F}] = [\mathcal{L}_\lambda \boxtimes \mathcal{F}] \/. \]
Furthermore, this isomorphism is linear with respect to $K_G(pt) \subset K_B(pt)$. (Recall that $K_G(pt) = K_B(pt)^W$.)
\end{lemma}
If one applies this Lemma for $X=pt$, then one obtains the isomorphism of $K_G(pt)$-algebras $R(T) \simeq K_B(pt) \simeq K_G(G/B)$, sending $e^\lambda \mapsto [\mathcal{L}_\lambda]$. By a theorem of Pittie, since $G$ is simply connected, $K_G(G/B)$ is a free $R(G)$-module; see~e.g.~\cite[Cor. 6.1.8]{chriss2009representation}. 

In general, by the K{\"u}nneth formula \cite[Thm.~5.6.1]{chriss2009representation}, we have $K_G(pt)=R(G)$-module isomorphisms $K_B(X) \simeq K_G(G/B \times X) \simeq K_G(G/B) \otimes_{R(G)} K_G(X) \simeq R(T)  \otimes_{R(G)} K_G(X)$. As explained in 
Thm.~6.1.22 of {\em loc.~cit.} the  isomorphism of the end terms is given by
\begin{equation}\label{E:kunneth} R(T) \otimes_{R(G)} K_G(X) \to K_B(X); \quad e^\lambda \otimes [\mathcal{F}] \mapsto [\C_\lambda] \otimes [\mathcal{F}] \/, \end{equation} obtained by regarding a $G$-linearized sheaf $\mathcal{F}$ as $B$-linearized. 

Next we apply these facts to define a convolution operator as in ~\cite[\S 5.2.20]{chriss2009representation}. Let $\pi_{i,j}$ be the projection from $G/B\times G/B\times X$ to the $i$ and $j$ factors, where $i\neq j\in \{1,2,3\}$. For any $a \in K_G(G/B\times G/B)$ and $b \in K_G(G/B \times X)$ define their convolution product to be 
\[a\star b:=(\pi_{1,3})_*(\pi_{1,2}^*(a)\otimes \pi_{2,3}^*(b)).\] 
For any simple root $\alpha_i$, consider the closure of the $G$-orbit: $Y_i := \overline{G .(1,e_{s_i})} \subset G/B \times G/B$. For any $B$-linearized sheaf $\mathcal{F}$ on $X$, the left convolution by the class $[\mathcal{O}_{Y_i}]\in K_G(G/B\times G/B)$ gives an operator $\widetilde{\delta}_i \in End_{K_G(pt)} (K_B(X))$ defined by: 
\[ \widetilde{\delta}_i[\mathcal{F}] = \Psi^{-1}([\mathcal{O}_{Y_i}] \star [G \times^B \mathcal{F}]) \/. \] 

Since $X$ is a $G$-variety, the automorphism $\Phi_w$ from \S \ref{sec:twoWactions} induces an action of the Weyl group $W$ on both $K_B(X)$ and $K^B(X)$, denoted by $w^L$ (for $w \in W$). This action satisfies $w^L[E \otimes \mathcal{F}] = w^L[E] \otimes w^L[\mathcal{F}]$ for $[E] \in K^B(X)$ and $[\mathcal{F}] \in K_B(X)$. By \cite[Thm.~6.1.22]{chriss2009representation}, there the natural inclusion gives an isomorphism $K_B(X)^W \simeq K_G(X)$. Then we may also define a generalization of the operator $\delta_i$ from \S \ref{sec:KDem}: \[ \delta_i = \frac{id}{1- e^{\alpha_i}} -\frac{ e^{\alpha_i} }{1- e^{\alpha_i}} s_i^L\/ \in End_{R(G)}(K_B(X)) \/. \] Denote by $\partial_i' $ the Demazure operator $\delta_i$ in the case $X=pt$; thus $\partial_i'$ acts on $R(T)= K_B(pt)$. Our main result is the following:
\begin{prop}\label{prop:eqdelta} (a) There is an equality of $K^G(X)$-linear endomorphisms $\delta_i= \widetilde{\delta}_i$ in $End_{R(G)} (K_B(X))$.

(b) Under the K{\"u}nneth isomorphism $K_B(X) \simeq R(T) \otimes_{R(G)} K_G(X)$ from \eqref{E:kunneth}, $\delta_i = \partial_{i}' \otimes 1$, where $1$ denotes the identity.\end{prop}
\begin{proof} By the K{\"u}nneth formula from \eqref{E:kunneth}, it suffices to check the claim on elements of the form $[\C_\lambda \otimes \mathcal{F}]\in K_B(X)$, where $\mathcal{F}$ is $G$-linearized. By Lemma \ref{lemma:product}, under the sequence of isomorphisms $K_B(X) \simeq K_G(G \times^B X) \simeq K_G(G/B \times X)$, the class $[\C_\lambda \otimes \mathcal{F}]$ is sent to $[\mathcal{L}_\lambda \boxtimes \mathcal{F}] \in K_G(G/B \times X)$. Recall the fibre diagram 
\[ \xymatrix{ Y_i \simeq G/B \times_{G/P_i} G/B \ar[d]^{pr_1} \ar[rrr]^{pr_2} & &&G/B \ar[d]^{\pi_i} \\ G/B \ar[rrr]^{\pi_i} &&& G/P_i } \] 
where $P_i$ is the minimal parabolic. It follows from the definition $\partial_i = \pi_i^* (\pi_i)_*$ that for any $a \in K_B(G/P_i)$, $\partial_i(a) = (pr_2)_*([\cO_{Y_i}] \cdot pr_1^*(a) )$. From this and the definition of the convolution (and the judicious use of the projection formula) we deduce that \[ 
[\cO_{Y_i}] \star [\mathcal{L}_\lambda \boxtimes \mathcal{F}] = \partial_i[\mathcal{L}_\lambda] \boxtimes [\mathcal{F}] \/. \] Since the inverse $\Psi^{-1}$ is given by its restriction to the fibre $pr^{-1}(1.B)$, \[  \widetilde{\delta}_i[\C_\lambda \otimes \mathcal{F}] = \iota^* (\partial_i[\mathcal{L}_\lambda]) \otimes \iota^* m^* [\mathcal{F}] = (\partial_i[\mathcal{L}_\lambda]|_{1.B}) \otimes [\mathcal{F}] \/, \] because $m \circ \iota = id_{X}$. Here $\partial_i[\mathcal{L}_\lambda]|_{1.B} \in K_B(pt)$ denotes the localization of $\partial_i[ \mathcal{L}_\lambda] \in K_B(G/B)$ at the fixed point $1.B$. By definition, \begin{equation}\label{equ:partiali}
 \partial_i[\mathcal{L}_\lambda]|_{1.B} = \frac{e^\lambda - e^{\alpha_i} [\mathcal{L}_\lambda]|_{s_i}}{1-e^{\alpha_i}} = \frac{e^\lambda - e^{\alpha_i} e^{s_i(\lambda)}}{1- e^{\alpha_i}}  = \partial_i'(e^\lambda)\/, 
\end{equation}
where the last equality is by definition of $\partial_i'$. It follows that \[ \partial_i[\mathcal{L}_\lambda]|_{1.B} \otimes  \mathcal{O}_{X} = \delta_i[\C_\lambda \otimes \mathcal{O}_{X} ]\/. \] Since $\delta_i[\C_\lambda \otimes \mathcal{F}] = \delta_i[\C_\lambda \otimes \cO_X] \otimes [\mathcal{F}]$ (from definition, since $[\mathcal{F}] \in K_G(X)$), this proves the claim.  
\end{proof}
The analogue of the operators $\delta_i$ in $H^*_T(X)$ appeared in Peterson's notes \cite{peterson}. Proposition \ref{prop:eqdelta}(b) implies that $\delta_i$ is also equal to the operator denoted by $\delta_{\alpha_i,X}$ in \cite[\S 6]{harada.landweber.sjamaar:divided}.

\subsection{Left and Right Demazure operators for $G/B$}\label{appex:two} {
In this section we take $X=G/B$, and we use the previous results of this section to give two additional constructions of the left and right Demazure operators for the complete flag variety $G/B$. The constructions have the merit of making clear the symmetry between the left and right operators. Recall that $K_B(G/B) = K_T(G/B)$.}

\subsubsection{Convolution approach} Next we
construct the left and the right Demazure operators as convolution operators. As in the proof of the Proposition \ref{prop:eqdelta}, observe that for any $G$-linearized sheaf $\mathcal{F}$ on $G/B$, and for any torus weight $\lambda$, 
\begin{equation*}\label{E:del}
[\mathcal{O}_{Y_i}]\star [\mathcal{L}_\lambda\boxtimes\mathcal{F}]=\partial_i[\mathcal{L}_\lambda]\boxtimes[\mathcal{F}] \/; \quad 
[\mathcal{L}_\lambda \boxtimes\mathcal{F}]\star [\mathcal{O}_{Y_i}]=[\mathcal{L}_\lambda]\boxtimes\partial_i[\mathcal{F}] \/.
\end{equation*} By Lemma \ref{lemma:product}, these convolution products give the 
$K_G(pt)$-linear endomorphisms of $K_B(G/B)$ defined by: 
\[ \widetilde{\delta}_i[\C_\lambda \otimes \mathcal{F}] :=\Psi^{-1}(\partial_i[\mathcal{L}_\lambda] \boxtimes [\mathcal{F}])\/; \quad 
\widetilde{\partial}_i[\C_\lambda \otimes \mathcal{F}] :=\Psi^{-1}([\mathcal{L}_\lambda] \boxtimes \partial_i[\mathcal{F}]) \/. \]
\begin{prop}\label{prop:convos} With notation as in \S \ref{sec:KDem}, $\widetilde{\delta}_i = \delta_i$ and $\widetilde{\partial}_i = \partial_i$.\end{prop}
\begin{proof} The first equality is proved in Proposition \ref{prop:eqdelta}. By K{\"u}nneth formula, it suffices to check the second 
equality on elements $[\C_\lambda \otimes \mathcal{F}]\in K_B(G/B)$, where $\mathcal{F}$ is $G$-linearized.
Then from Lemma \ref{lemma:product}: \[ \widetilde{\partial}_i [\C_\lambda \otimes \mathcal{F}] =  \iota^*([\mathcal{L}_\lambda] \boxtimes \partial_i[\mathcal{F}] )= \iota^* [\mathcal{L}_\lambda] \otimes \iota^* m^* \partial_i[\mathcal{F}] =  [\mathbb{C}_\lambda] \otimes \partial_i[\mathcal{F}] = \partial_i [\C_\lambda \otimes \mathcal{F}] \/, \] using that $\partial_i$ is $K_B(pt)$-linear, and that $m \circ \iota = id_{G/B}$. \end{proof}

\subsubsection{Atiyah-Hirzebruch approach} {Denote by $\overline{\Psi}_P$ the isomorphism
from \eqref{E:kunneth} for $X=G/P$.} Then
\[{\overline{\Psi}_{B}}:K_T(G/B)\xrightarrow{\sim} K_G(G/B\times G/B)=R(T)\otimes_{{R(G)}}R(T),\]
sends $[\mathbb{C}_\lambda] \otimes [\calL_\mu]$ to $e^\lambda\otimes e^\mu$. In the literature this is also referred as the Atiyah-Hirzebruch isomorphism. Under the isomorphism $\overline{\Psi}_{B}$, the left and right actions $w^L, w^R$ from \S \ref{sec:KDem} correspond to the natural action of the Weyl group $W$ on the left or the right components of $R(T) \otimes_{R(G)} R(T)$. To be precise, from the definition of $\overline{\Psi}_B$, it follows that there are commutative diagrams:
\[ \begin{tikzcd}
K_T(G/B) \arrow{r}{\overline{\Psi}_{B}} \arrow[swap]{d}{w^L} & R(T)\otimes_{R(G)}R(T) \arrow{d}{w \otimes 1} \\%
K_T(G/B) \arrow{r}{\overline{\Psi}_{B}}& R(T)\otimes_{R(G)}R(T)
\end{tikzcd} \quad 
 \begin{tikzcd}
K_T(G/B) \arrow{r}{\overline{\Psi}_{B}} \arrow[swap]{d}{w^R} & R(T)\otimes_{R(G)}R(T) \arrow{d}{1\otimes w} \\%
K_T(G/B) \arrow{r}{\overline{\Psi}_{B}}& R(T)\otimes_{R(G)}R(T)
\end{tikzcd}
\]
The localization of $[\mathbb{C}_\lambda] \otimes [\calL_\mu]$
at $w \in W$ equals to $e^\lambda w(e^\mu)$, and this localization corresponds to the composition $\widetilde{m} \circ (id \otimes w)$, where 
$\widetilde{m}:R(T) \otimes_{R(G)} R(T) \to R(T)$ is the multiplication. Since $[\mathbb{C}_\lambda] \otimes [\calL_\mu]$ generate 
$K_T(G/B)$ as a $K_T(pt)$-algebra, this gives a commutative diagram:
\[
 \begin{tikzcd}
K_T(G/B) \arrow{r}{\overline{\Psi}_{B}} \arrow[swap]{d}{\iota_w^*} & R(T)\otimes_{R(G)}R(T) \arrow{d}{\widetilde{m} \circ (id \otimes w)} \\%
R(T) \arrow{r}{=}& R(T)
\end{tikzcd}
\]
One can also define the analogues of the left and right Demazure operators as follows.
Recall the operator on $R(T)= K_B(pt)$:
\[\partial'_i = \frac{id}{1- e^{\alpha_i}} -\frac{ e^{\alpha_i} }{1- e^{\alpha_i}}s_i.\]
\begin{prop} With notation as in \S \ref{sec:KDem}, and
under the isomorphism ${\overline{\Psi}_{B}}$, 
\[\delta_i=\partial'_i\otimes 1, \textit{ \quad and \quad}\partial_i=1\otimes \partial_i'.\]
\end{prop}
\begin{proof}
The first identity is already proved in Proposition \ref{prop:eqdelta}(b). To prove the second one, for any torus weights $\lambda$ and $\mu$, 
{\[ \overline{\Psi}_{B} \partial_i ([\mathbb{C}_\lambda] \otimes [\calL_\mu]) = \overline{\Psi}_{B}([\mathbb{C}_\lambda] \otimes \partial_i [\calL_\mu]) = e^\lambda \otimes (\partial_i [\calL_\mu]|_{1.B}) = e^\lambda \otimes \partial_i'(e^\mu) \/, \]
where the last equality follows from the equation \eqref{equ:partiali}.}
\end{proof}
By utilizing these operators one may write expressions for the more general Demazure-Lusztig operators from \S \ref{sec:KDem}. For instance, the operator corresponding to $\mathcal{T}_i^R$ is $(1+y e^{\alpha_i}) (1 \otimes \partial_i') - id$.

We now turn to functorial properties. For any standard parabolic subgroup $P$ containing $B$, there is an isomorphism
{\[ \overline{\Psi}_{P}:K_T(G/P)\xrightarrow{\sim} K_G(G/B\times G/P)=R(T)\otimes_{{R(G)}} K_G(G/P) \/.\]
The localization at $1.P$ gives an identification ${K_G(G/P) \simeq R(T)^{W_P} \simeq R(P)}$ 
as $K_G(pt)$-algebras, where $W_P$ is the Weyl group of $P$.} Let $w_0^P$ be the longest element in $W_P$.

\begin{prop} The natural projection $\pi:G/B\rightarrow G/P$ gives the following commutative diagrams:
\[ \begin{tikzcd}
K_T(G/B) \arrow{r}{\overline{\Psi}_{B}} \arrow[swap]{d}{\pi_*} & R(T)\otimes_{R(G)}R(T) \arrow{d}{1\otimes \partial'_{w_0^P}} \\%
K_T(G/P) \arrow{r}{\overline{\Psi}_{P}}& R(T)\otimes_{R(G)}R(T)^{W_P}
\end{tikzcd} \quad 
 \begin{tikzcd}
K_T(G/P) \arrow{r}{\overline{\Psi}_{P}} \arrow[swap]{d}{\pi^*} & R(T)\otimes_{R(G)}R(T)^{W_P} \arrow{d}{1\otimes \incl} \\%
K_T(G/B) \arrow{r}{\overline{\Psi}_{B}}& R(T)\otimes_{R(G)}R(T)
\end{tikzcd}
\]
where $\incl$ is the inclusion of $R(T)^{W_P} \hookrightarrow R(T)$.\end{prop}

\begin{proof} The commutativity of the second diagram follows from the definition of $\overline{\Psi}_{P}$ 
and the fact that $\pi^*[G \times^P V] = [G \times^B V] $ for any $P$-representation $V$, since these classes generate 
$K_T(G/P)$ over $K_T(pt)$. The commutativity of the first diagram can be proved as follows. If $P=G$, then 
$\pi_*(\calL_\lambda)=\chi(G/B,\calL_\lambda)=\partial_{w_0}'(e^\lambda)$, 
by the Demazure character formula (see e.g.~\cite[Theorem 5.1]{MSA:whittaker} 
for a simple proof and generalizations). Let now $P$ be arbitrary
and $[\mathcal{L}_\lambda] \in K_G(G/B)$. Then $\pi_*[\mathcal{L}_\lambda] = [\mathbb{V}_\lambda]$,
where $\mathbb{V}_\lambda$ is the (virtual) homogeneous bundle
$G \times^P \chi(P/B, \mathcal{L}_\lambda)$. But $P/B$ is the flag manifold $L/B_L$ where $L$ is the Levi subgroup 
of $P$ and $B_L = B \cap L$ is the Borel subgroup determined by $B$. The Weyl group of $L$ is $W_P$.~By the $P=G$ case, 
the representation $\chi(P/B, \mathcal{L}_\lambda)$ has character $\partial_{w_0^P}'(\lambda)$. Then the commutativity follows
from the fact that the isomorphism $K_G(G/P) \simeq R(T)^{W_P}$ is given by localization at $1.P$.
\end{proof}  


\bibliographystyle{halpha}
\bibliography{leftdiv}

\end{document}